 \documentclass[12pt]{amsart}

\usepackage{amsmath}
\usepackage{amssymb}
\usepackage{mathdots}
\usepackage{mathrsfs}
\usepackage{amsthm}
\usepackage{tikz}
\usepackage{adjustbox}
\usepackage[all]{xy}

\numberwithin{equation}{section}

\newtheorem{thm}{Theorem}[section]

\newtheorem{prop}[thm]{Proposition}
\newtheorem{cor}[thm]{Corollary}
\theoremstyle{remark}
\newtheorem{rem}[thm]{Remark}
\newtheorem{exam}[thm]{Example}
\newtheorem{exam-nota}[thm]{Example-Notation}
\newtheorem{rem-nota}[thm]{Remark-Notation}
\newtheorem{nota}[thm]{Notation}
\theoremstyle{definition}
\newtheorem{dfn}[thm]{Definition}

\newtheorem{dfn-nota}[thm]{Definition-Notation}

\newtheorem{dfn-lem}[thm]{Lemma-Definition}
\newtheorem{dfn-prop}[thm]{Proposition-Definition}

\newcommand{\beqa}{\begin{eqnarray*}}
\newcommand{\eeqa}{\end{eqnarray*}}

\newcommand{\fa}{\mbox{${\mathfrak a}$}}

\newcommand{\fk}{\mbox{${\mathfrak k}$}}
\newcommand{\fg}{\mbox{${\mathfrak g}$}}

\newcommand{\fh}{\mbox{${\mathfrak h}$}}

\newcommand{\fp}{\mbox{${\mathfrak p}$}}
\newcommand{\fr}{\mbox{${\mathfrak r}$}}
\newcommand{\fb}{\mbox{${\mathfrak b}$}}

\newcommand{\supp}{{\rm supp}}

\newcommand{\eps}{\epsilon}

\newcommand{\PR}{\mbox{${\mathbb P}$}}

\newcommand{\C}{\mbox{${\mathbb C}$}}

\newcommand{\Ad}{{\rm Ad}}

\newcommand{\fgl}{\mathfrak{gl}}

\newcommand{\B}{\mathcal{B}}

\newcommand{\ms}{m(s_{\alpha})}

\newcommand{\he}{\hat{e}}

\newcommand{\F}{\mathcal{F}}

\newcommand{\G}{\mathcal{G}}
\newcommand{\Borbitspace}{B_{n-1}\backslash\B_{n}}
\newcommand{\Sh}{\mathcal{S}h}
\newcommand{\calO}{\mathcal{O}}

\newcommand{\Sp}{\mathcal{S}p}

\AtBeginDocument{%
   \def\MR#1{}
}\title{Orbits on a product of two flags and a line and the Bruhat order, I}

\author[M. Colarusso]{Mark Colarusso}
\address{Department of Math and Stats, University of South Alabama, Mobile, AL, 36608}
\email{mcolarusso@southalabama.edu}

\author[S. Evens]{Sam Evens}
\address{Department of Mathematics, University of Notre Dame, Notre Dame, IN, 46556}
\email{sevens@nd.edu}
\begin{document}
\maketitle
\tableofcontents
\begin{abstract}
Let $G=GL(n)$ be the $n\times n$ complex general linear group and let $\B_{n}$ be its flag variety.  The standard Borel subgroup $B$ of upper triangular matrices acts on the product $\B_{n}\times \mathbb{P}^{n-1}$ diagonally with finitely many orbits.  In this paper, we study the $B$-orbits on the subvarieties $\B_{n}\times \mathcal{O}_{i}$, where $\mathcal{O}_{i}$ is the $B$-orbit on $\mathbb{P}^{n-1}$ containing the line through the origin in the direction of the $i$-th standard basis vector of $\C^{n}$.   For each $i=1,\dots, n$, we construct a bijection between $B$-orbits on $\B_{n}\times\mathcal{O}_{i}$ and certain pairs of Schubert cells in $\B_{n}\times\B_{n}$.  We also show that this bijection can be used to understand the Richardson-Springer monoid action on such $B$-orbits in terms of the classical monoid action of the symmetric group on itself.  We also develop combinatorial models of these orbits and use these models to compute exponential generating functions for the sequences $\{|B\backslash(\B_{n}\times\mathcal{O}_{i})|\}_{n\geq 1}$ and $\{|B\backslash (\B_{n}\times \mathbb{P}^{n-1})|\}_{n\geq 1}$.  In the sequel to this paper, we use the results of this paper to construct a correspondence between $B$-orbits on $\B_{n}\times\mathbb{P}^{n-1}$ and a collection of $B$-orbits on the flag variety $\B_{n+1}$ of $GL(n+1)$ and show that this correspondence respects closures relations and preserves monoid actions.   As a consequence both closure relations and monoid actions on the set of all $B$-orbits on $\B_{n}\times\mathbb{P}^{n-1}$ can be understood via the Bruhat order by using our results in \cite{Shpairs}. 

\end{abstract}

\section{Introduction}
Let $G=GL(n)$ be the $n\times n$ complex general linear group, let $\B_{n}$ be its flag variety, and let $B_{n-1}$ be the standard Borel subgroup of $(n-1)\times (n-1)$ invertible upper triangular matrices embedded in the top lefthand corner of $G.$ It is well-known that $B_{n-1}$ has finitely many orbits on $\B_{n}.$  In our paper \cite{Shpairs}, we found a bijection between $B_{n-1}$-orbits on $\B_n$ and certain pairs of Schubert cells in $\B_n \times \B_n$ which preserves the poset structure on these orbits given by orbit closures, and intertwines monoid actions from our paper \cite{CE21I} with well-understood monoid actions on Schubert cells.  These results make understanding the relationships between orbits significantly more transparent.

The purpose of this paper and its sequel is to extend our results from \cite{Shpairs} to a much larger class of orbits.   Indeed, the upper triangular Borel subgroup $B$ of $G$ acts diagonally on $\B_n \times \mathbb{P}^{n-1}$ with finitely many orbits.   These orbits have been studied  by Magyar \cite{Magyar} and Travkin \cite{Travkin} and are of broad interest since they are related to the study of mirabolic $\mathcal{D}$-modules which arise in the analysis of category $\mathcal{O}$ for rational Cherednik algebras \cite{FG}. 
   The $B$-orbits on $\B_n \times \mathbb{P}^{n-1}$ are partitioned according to their projections to the $B$-orbits  on $\mathbb{P}^{n-1},$ which are of the form $\mathcal{O}_i:=B\cdot [e_{i}]$ where $[e_{i}]$ is the line through the $i$-th standard basis vector $e_{i}.$   The  $B$-orbits on $\B_n \times \mathcal{O}_{i}$ correspond naturally to the $S_i$-orbits on $\B_n$ where $S_i$ is the stabilizer in $B$ of the line $[e_{i}].$  Thus, results about the collection of $S_i$-orbits for $i=1, \dots, n$ imply results about $B$-orbits on $\B_n \times \mathbb{P}^{n-1}.$   When $i=n$, $S_{n}=B_{n-1}\times Z$ where $Z$ is the centre of $GL(n)$.  In this paper, we give a bijection denoted $\Sh_i$  between $S_i$-orbits on $\B_n$ and certain pairs of Schubert cells in $\B_n \times \B_n$ defined using the index $i$ (Theorem \ref{thm:ShpairsforSi} and Remark \ref{r:geoshar}), generalizing one of the main results of \cite{Shpairs} concerning $B_{n-1}$-orbits to $S_{i}$-orbits for any $i$.  In Section \ref{s:monoid}, we generalize the monoid actions from our paper \cite{CE21I} to the setting of $S_{i}$-orbits on $\B_{n}$ and establish the compatibility of these monoid actions with the usual monoid actions on Schubert cells (Theorem \ref{thm:intertwine}).  We prove that the closure relations on $S_i$-orbits on $\B_n$ are determined using the Bruhat order in the sequel to this paper.   In the sequel, we use results of this paper to provide a correspondence between $B$-orbits on $\B_n \times \mathbb{P}^{n-1}$ and a collection of $B$-orbits on the flag variety $\B_{n+1}$ of $GL(n+1),$ and show that our correspondence preserves closure relations and respects the above monoid actions.   As a consequence, our results from \cite{Shpairs} imply that we can understand the closure relations together with the monoid actions for $B$-orbits on $\B_n \times \mathbb{P}^{n-1}$ using the well-understood Bruhat order on the two-fold product $\Sigma_{n+1} \times \Sigma_{n+1}$ of the symmetric group.

  In more detail, consider the set $G_{\Delta}\backslash(\B_{n}\times\B_{n}\times\mathbb{P}^{n-1})$ of $G$-diagonal orbits on $\B_{n}\times\B_{n}\times\mathbb{P}^{n-1}.$ 
  Then we have natural correspondences
\begin{equation}\label{eq:allthecorres}
\coprod_{i=1}^{n} S_{i}\backslash\B_{n}\longleftrightarrow \coprod_{i=1}^{n} B\backslash (\B_{n}\times\mathcal{O}_{i})\longleftrightarrow G_{\Delta}\backslash(\B_{n}\times\B_{n}\times\mathbb{P}^{n-1}).
\end{equation}

\noindent Let $B^{i}\subset G$ be the Borel subgroup which stabilizes the flag 
\begin{equation}\label{eq:introiflag}
\mathcal{E}^{i}:=(\mathcal{E}_{1}^{i}\subset \dots\subset \mathcal{E}_{n}^{i})
\end{equation} 
where for $\ell = 1, \dots, i-1,$  $\mathcal{E}_{\ell}^{i}=\mbox{span}\{e_{i}, \, e_{1},\dots, e_{\ell-1}\}$ and for $\ell \geq i,$  $\mathcal{E}_{\ell}^{i}=\mbox{span}\{e_{1},\dots, e_{\ell}\}.$
\begin{thm}\label{thm:introiSh}[Theorem \ref{thm:ShpairsforSi}]
The $S_{i}$-orbits on $\B_{n}$ are precisely the non-empty intersections of $B$ and $B^{i}$-orbits on $\B_{n}$.
\end{thm}
Theorem \ref{thm:introiSh} allows us to generalize the notion of a Shareshian pair from \cite{Shpairs} to this setting.  Let $Q\in S_{i}\backslash\B_{n}$, let $\mathcal{E}$ be the standard upper triangular flag in $\C^{n}$, and let $W$ be the Weyl group of $G.$  Then by the Bruhat decomposition $B\cdot Q=B\cdot w(\mathcal{E})$ for some $w\in W$ and $B^{i}\cdot Q=B^{i}\cdot u^{i}(\mathcal{E}^{i})$ for some $u^{i}\in W$.  
We define the \emph{$i$-Shareshian} pair of $Q$ to be 
\begin{equation}\label{eq:introiShpair}
\Sh_{i}(Q)=(w,u^{i})\Leftrightarrow Q=B\cdot w(\mathcal{E})\cap B^{i}\cdot u^{i}(\mathcal{E}^{i})\neq\emptyset,
\end{equation}
and we call the map 
\begin{equation}\label{eq:introiShmap}
\Sh_{i}: S_{i}\backslash\B_{n}\to W\times W; \; Q\mapsto \Sh_{i}(Q),
\end{equation}
the $i$-\emph{Shareshian map}.

Theorem \ref{thm:introiSh} is proven by constructing a canonical set of representatives for the $S_{i}$-orbits on $\B_{n}$ in terms of flags, which we call \emph{flags in $i$-standard form} (Theorem \ref{thm:istd}).  Using the fact that every $S_{i}$-orbit on $\B_{n}$ contains a unique flag in $i$-standard form, we can generalize our combinatorial model for $\Borbitspace$ developed in Section 3 of \cite{CE21II} to $S_{i}\backslash\B_{n}$.  In \emph{loc. cit.}, the set of $B_{n-1}$-orbits on $\B_{n}$ are parameterized by partitions into lists (PILs) of the set $\{1,\dots, n\}$, where a list is any ordered subset of $\{1,\dots, n\}$.  We show that 
$S_{i}$-orbits are parameterized by pairs $(\sigma,\Sigma)$ where $\Sigma$ is a PIL of $\{1,\dots, n\}$ and $\sigma\in \Sigma$ is a list of size $n-i$ (Proposition \ref{p:countSi}).  We then use this result to compute exponential generating functions for the sequences $\{|S_{i}\backslash \B_{n}|\}_{n\geq 1}$ and $\{|G_{\Delta}\backslash(\B_{n}\times\B_{n}\times\mathbb{P}^{n-1})|\}_{n\geq 1}$ (Proposition \ref{p:egfforSi} and Theorem \ref{thm:bigegf} respectively).  We also show that $|S_{i}\backslash\B_{n}|$ is computed by a very simple recursion (Corollary \ref{c:Sirecursion}). 

In \cite{CE21I} and \cite{CE21II}, we studied an \emph{extended monoid} action on $\Borbitspace$ by simple roots of $\fk:=\fgl(n-1)$ and $\fg$.  This monoid action plays an important role in understanding the closure relations on $B_{n-1}\backslash\B_{n}$.  In this paper, we study an \emph{extended monoid action} on $S_{i}\backslash \B_{n}$ by a set of roots $\mathfrak{S}_i\coprod \Pi_{\fg}$, where $\mathfrak{S}_i\subset \Pi_{\fg}$ is a certain subset of simple roots.  In the case where $i=n$ the subset $\mathfrak{S}_i$ can be naturally identified with the standard simple roots for $\fk$.   The Weyl group $W$ has well-studied left and right monoid actions by simple roots of $\fg$, whose properties are closely tied to the Bruhat order on $W$.  We use this monoid action on $W$ to define left and right monoid actions by the simple roots $\mathfrak{S}_i$ and $\Pi_{\fg}$ respectively on $W\times W$.
\begin{thm}\label{thm:intromonoid}(Theorem \ref{thm:intertwine})
 The $i$-Shareshian map in Equation (\ref{eq:introiShmap}) intertwines the generalized extended monoid action on $S_{i}$-orbits on $\B_{n}$ with the classical monoid action on $W\times W$ described above, i.e., for a simple root $\alpha\in \mathfrak{S}_i\coprod\Pi_{\fg}$ and a $S_{i}$-orbit $Q$, 
 \begin{equation}\label{eq:intromonoids}
 \Sh_{i}(\ms*Q)=\ms*\Sh_{i}(Q).
 \end{equation}
 \end{thm}

In the sequel, we will use Theorem \ref{thm:intromonoid} and results in Section 6 of \cite{CE21II} to show that every $S_{i}$-orbit on $\B_{n}$ is related in the weak order defined by the generalized extended monoid action to a zero dimensional $S_{i}$-orbit, implying that the closure ordering on $S_{i}\backslash \B_{n}$ is given by the standard order of Richardson and Springer. We will also show that 
the closure ordering on $S_{i}\backslash\B_{n}$ is given by the $i$-Shareshian order, i.e.,
\begin{equation}\label{eq:iShorder}
Q^{\prime}\subset \overline{Q}\Leftrightarrow \Sh_{i}(Q^{\prime})\leq \Sh_{i}(Q).
\end{equation}
Here $\Sh_{i}(Q)$ is given in (\ref{eq:introiShpair}) and the ordering on the $i$-Shareshian pairs is given by the product of Bruhat orders on $W\times W$.  In the second factor, the Bruhat order is twisted by using the simple reflections with respect to the Borel subalgebra $\fb^{i}$ which stabilizes the flag $\mathcal{E}^{i}$ in Equation (\ref{eq:introiflag}).  We prove this result in the sequel by establishing a bijection $\theta$ between $B$-orbits on $\B_n \times \mathbb{P}^{n-1}$  and $B$-orbits on an open $G$-stable subvariety $\mathfrak{X}$ of the flag variety $\B_{n+1}$ of $\fgl(n+1).$   Using our results from \cite{Shpairs}, we show that if $Q, Q^{\prime} \in B\backslash(\B_n \times \mathbb{P}^{n-1}),$ then $Q^{\prime} \subset \overline{Q}$ if and only if $\Sh_{n+1}(\theta(Q^{\prime})) \subset \overline{\Sh_{n+1}(\theta(Q))}$ and similar results hold for the monoid action.   As a consequence, we can understand the poset structure and monoid actions on $B$-orbits on $\B_n \times \mathbb{P}^{n-1}$ in terms of the well-understood analogous structures on $\Sigma_{n+1} \times \Sigma_{n+1}$.

This paper is organized as follows.  After introducing notation in Section \ref{s:prelim}, in Section \ref{s:reps}, we show each $S_i$-orbit is the intersection of a $B$-orbit and a $B^i$-orbit by using the notion of $i$-standard form for flags.  We use this assertion to define the $i$-Shareshian map from $S_i$-orbits on the flag variety to pairs of Weyl group elements and characterize the image.  In Section \ref{s:models}, we introduce the variant of PIL's needed to count $S_i$-orbits on $\B_n.$  As a consequence, we compute the exponential generating functions for the sequence of $S_{n-i}$-orbits on $\B_n$ and for $G_{\Delta}$-orbits on $\B_n \times \B_n \times \mathbb{P}^{n-1}$ and deduce formulas which simplify counting of $S_i$-orbits on $\B_n.$   Section \ref{s:monoid} is focused on constructing the extended monoid action for $S_{i}$-orbits on $\B_{n}$ and proving the compatibility of these monoid actions with classical monoid action on pairs of Weyl group elements.

We would like to thank Jeb Willebring for useful discussions regarding Section \ref{s:models} and John Shareshian for useful suggestions.

\section{Preliminaries and Notation}\label{s:prelim}

In this paper, all algebraic varieties are by convention complex algebraic varieties, and similarly with Lie algebras.  Let $G=GL(n)$ be the complex general linear group and $\fg=\fgl(n)$ be its Lie algebra.  Embed $GL(n-1)$ in $G$ as matrices fixing the standard basis vector $e_{n}$. We let $\fh\subset \fg$ denote the standard Cartan subalgebra of diagonal matrices and let $H\subset G$ be the corresponding algebraic group.  We let $\eps_{j}\in\fh^{*}$ be the linear functional on $\fh$ which acts on $x= \mbox{diag}[h_{1},\dots, h_{j},\dots, h_{n}]\in\fh$ by $\eps_{j}(x)=h_{j}$.  For an algebraic group $A$ with Lie algebra $\fa$, we denote the adjoint action of $A$ on $\fa$ by $\Ad(g)x$ for $g\in A$ and $x\in\fa$.  Abusing notation, we also denote the action of $A$ on itself by conjugation by $\Ad$, so that $\Ad(g)h:=ghg^{-1}$ for $g,\, h\in A$.  

 Throughout the paper, we identify the variety $\B_{n}$ of Borel subalgebras of $\fg$ with the variety
 $\mbox{Flag}(\C^{n})$ of full
flags in $\C^{n}$.  The group $G$ acts on $\B_{n}$ via the adjoint action.  Let $W=N_{G}(H)/H$ be the Weyl group of $G$.  If $\tilde{\fb}\in \B_{n}$ is an $H$-stable Borel subalgebra of $\fg$ which stabilizes the flag $\tilde{\mathcal{E}}$, then $\Ad(\dot{w}) \tilde{\fb}$ is independent of the choice of representative $\dot{w}$ of $w\in W$.  We write $\Ad(w)\tilde{\fb}$ or simply $w(\tilde{\fb})$ for the Borel subalgebra stabilizing the flag $w(\tilde{\mathcal{E}})$.  

 We make heavy use of the following notation for flags throughout the paper.  
\begin{nota}\label{nota:standard} Let 
   $$
  \mathcal{F}=(V_{1}\subset V_{2}\subset\dots\subset V_{i}\subset V_{i+1}\subset \dots)
   $$
be a flag in $\C^{n}$, with $\dim V_{i}=i$ and $V_{i}=\mbox{span}\{v_{1},\dots, v_{i}\}$, with each $v_{j}\in\C^{n}$.  We will denote this flag $\mathcal{F}$ 
by
   $$
  \mathcal{F}=  (v_{1}\subset  v_{2}\subset\dots\subset v_{i}\subset v_{i+1}\subset\dots ). 
   $$
\end{nota}
\noindent For a vector $v\in\C^{n}$, we denote the line through $v$ containing the origin by $[v]\in\mathbb{P}^{n-1}$.


As mentioned in the Introduction, $G$ acts diagonally on the triple product $\B_{n}\times\B_{n}\times \mathbb{P}^{n-1}$ with finitely many orbits and we denote the set of these orbits by $G_{\Delta}\backslash (\B_{n}\times\B_{n}\times\mathbb{P}^{n-1})$.  These orbits are of course in one-to-one correspondence with 
diagonal $B$-orbits on the product $\B_{n}\times \mathbb{P}^{n-1}$ by  
\begin{equation}\label{eq:Gdeltacorres}
G_{\Delta}\backslash(\B_{n}\times\B_{n}\times\mathbb{P}^{n-1})\longleftrightarrow B\backslash (\B_{n}\times \mathbb{P}^{n-1}) \mbox{ given by } G_{\Delta}\cdot (\fb,\fb^{\prime},[v])\leftrightarrow B\cdot (\fb^{\prime}, [v]),
\end{equation}  
where $\fb$ is the Lie algebra of $B.$

To describe the $B$-orbits on $\mathbb{P}^{n-1}$, let $\{e_{1},\dots, e_{n}\}$ denote the standard basis of $\C^{n}$.  Then there are exactly $n$ $B$-orbits on $\mathbb{P}^{n-1}$, which are  given by $\mathcal{O}_{i}:=B\cdot[e_{i}]$ for $i=1,\dots, n$.  We can decompose the set $B\backslash (\B_{n}\times\mathbb{P}^{n-1})$ as
\begin{equation}\label{eq:orbitspacedecomp}
B\backslash (\B_{n}\times\mathbb{P}^{n-1})=\coprod_{i=1}^{n} B\backslash (\B_{n}\times\mathcal{O}_{i}).
\end{equation}
Fix an $i$ with $1\leq i \leq n$ and consider the set of $B$-orbits $B\backslash (\B_{n}\times\mathcal{O}_{i}).$  Let
$S_{i}:=\mbox{Stab}_{B}([e_{i}])$.  Then we have a one-to-one correspondence 
\begin{equation}\label{eq:Sicorres}
B\backslash (\B_{n}\times\mathcal{O}_{i})\longleftrightarrow S_{i}\backslash \B_{n} \mbox{ given by } B\cdot(\fb_1,[e_{i}])\longleftrightarrow S_{i}\cdot\fb_1.
\end{equation} 
Our main object of study in this paper is the set of orbits $S_{i}\backslash\B_{n}$.  
\begin{rem}\label{r:specialcase}
We note that if $i=1$, then $S_{1}=B$ and the $B$-orbits on $\B_{n}$ are given by the Bruhat decomposition and their closure ordering is given by the Bruhat ordering on $W$.  If $i=n$, then up to centre $S_{n}$ coincides with $B_{n-1}$ and the structure of $B_{n-1}\backslash\B_{n}$ is described extensively in \cite{Hashi}, \cite{CE21I}, \cite{CE21II}, \cite{Shpairs}, and other papers. 
\end{rem}

\section{Representatives for $S_{i}$-orbits and $i$-Shareshian pairs}\label{s:reps}
Let $S_{i}=\mbox{Stab}_{B}([e_{i}])$ as in the previous section.  Let 
$\fb^{i}\in\B_{n}$ be the Borel subalgebra which stabilizes the flag $\mathcal{E}^{i}$ in (\ref{eq:introiflag}):
\begin{equation}\label{eq:Fiflag}
\mathcal{E}^{i}=(e_{i}\subset e_{1}\subset\dots\subset e_{i-1}\subset e_{i+1}\subset \dots\subset e_{n}),
\end{equation}
and let $B^{i}$ be the corresponding Borel subgroup of $G.$
Our goal in this section is to prove the following result:
\begin{thm}\label{thm:ShpairsforSi}
Let $Q=S_{i}\cdot\fb^{\prime}$ be the $S_{i}$-orbit through $\fb^{\prime}\in\B_{n}$ and let $Q_{B}:=B\cdot\fb^{\prime}$ and $Q_{B^{i}}:=B^{i}\cdot\fb^{\prime}$ 
be the $B$ and $B^{i}$-orbits through $\fb^{\prime}\in\B_{n}$ respectively.  Then $Q=Q_{B}\cap Q_{B^{i}}$.  In particular, every $S_{i}$-orbit on $\B_{n}$ is the intersection of a $B$-orbit and a $B^{i}$-orbit on $\B_{n}$. 
\end{thm}

Note that $B\cap B^{i}=S_{i}$.  Therefore, if $Q_{1}$ (resp. $Q_{2}$) is a $B$-orbit  (resp. a $B^{i}$-orbit) on $\B_{n}$, and the intersection $Q_{1}\cap Q_{2}$ is non-empty, it contains an $S_{i}$-orbit $Q=S_{i}\cdot\fb^{\prime}$.  Thus, by Theorem \ref{thm:ShpairsforSi}, we deduce that $Q_{1}=Q_{B}$ and $Q_{2}=Q_{B^{i}}$, and
\begin{equation}\label{eq:intersect}
Q=Q_{B}\cap Q_{B^{i}}.
\end{equation}
Therefore, Theorem \ref{thm:ShpairsforSi} implies 
\begin{cor}\label{c:intersect}
The $S_{i}$-orbits on $\B_{n}$ are precisely the non-empty intersections of $B$ and $B^{i}$-orbits on $\B_{n}$. 
\end{cor}

To prove Theorem \ref{thm:ShpairsforSi}, we need to find a set of 
representatives for the $S_{i}$-orbits on $\B_{n}$ in terms of flags.  For this, we introduce the following definition (cf. Definition 4.1 of \cite{CE21II}).  
\begin{dfn}\label{d:istd}
\begin{enumerate}
\item For a standard basis vector $e_{j}\in\C^{n}$ with $1\leq j\leq i-1$, we define $\he_{j}:=e_{j}+e_{i}$ and refer to $\he_{j}$ as a hat vector of index $j.$
\item We say that a flag 
\begin{equation}\label{eq:basicflag}
\mathcal{F}:=(v_{1}\subset \dots \subset v_{i}\subset\dots\subset v_{n})
\end{equation}
in the flag variety $\B_{n}$ of $G$
is in \emph{$i$-standard form} if $v_{k}=e_{j_{k}}$ or $v_{k}=\he_{j_{k}}$ for all $k=1,\dots, n$, and 
$\mathcal{F}$ satisfies the following three conditions:
\begin{enumerate}
\item $v_k = e_i$ for some $k$.
\item If $v_{k}=e_{i}$, then $v_{\ell}=e_{j_{\ell}}$ for all $\ell>k$.
 \item If $\ell<k$ with $v_{\ell}=\he_{j_{\ell}}$ and $v_{k}=\he_{j_{k}}$, then $j_{\ell}>j_{k}$. 
\end{enumerate}
\end{enumerate}
\end{dfn}

\begin{rem}\label{r.glstandard}
If $\mathcal{F}=(v_{1}\subset \dots \subset v_{k}\subset\dots\subset v_{n})$ is a flag in $i$-standard form, then it follows easily that if $v_{k} = e_{j_{k}}$ or $\he_{j_{k}}$ for $k=1,\dots, n$, then $\{ j_1, \dots, j_n \} = \{ 1, \dots, n \}$ consists of $n$ distinct indices.
\end{rem}

\begin{thm}\label{thm:istd}
 Every $S_{i}$-orbit on $\B_{n}$ contains a unique flag in $i$-standard form.  
 \end{thm}
 
 \begin{rem}\label{r:usual}
Note that when $i=n$ the definition of a flag in $n$-standard form in Definition \ref{d:istd} coincides with the definition of a flag in standard form given in Definition 4.1 in \cite{CE21II}.  In this case, we will drop the ``n" from $n$-standard form.  We showed in Theorem 4.7 of \emph{loc. cit.} that each $B_{n-1}$-orbit on $\B_{n}$ contains a unique flag in standard form. 
\end{rem}

To prove Theorem \ref{thm:istd}, we first recall 
the parameterization of $B\backslash(\B_{n}\times\mathbb{P}^{n-1})$ given by Magyar.

\begin{dfn}\label{d:marked}
We define a \emph{decorated permutation} of $\{1,\dots, n\}$ to be a pair $(w,\Delta)$ where $w\in\Sigma_{n}$ and $\Delta=\{j_{1},\dots, j_{k}\}$ is a nonempty subset of $\{1,\dots,n\}$ which is a decreasing sequence for $w^{-1}$, i.e., $j_1 < \dots < j_{k}$ but $w^{-1}(j_{k})<\dots< w^{-1}(j_{1})$.
\end{dfn}

We denote the set of all decorated permutations of $\{1,\dots, n\}$ by $\widehat{\Sigma}_{n}$.  Given  $(w,\Delta) \in \widehat{\Sigma}_{n}$ with 
$\Delta=\{j_{1},\dots, j_{k}\}$ as above,
 we consider the $B$-orbit 
\[
{\mathcal{O}}_{(w,\Delta)}:= B\cdot (w(\fb), [v_{\Delta}]), \ \text{ where } v_{\Delta}=e_{j_1}+\dots + e_{j_k}.
\]

\begin{thm}\label{thm:magyar}\cite{Magyar}
Each $B$-orbit on $\B_{n}\times\mathbb{P}^{n-1}$ is of the form ${\mathcal{O}}_{(w,\Delta)}$ for a unique decorated permutation $(w,\Delta).$
\end{thm}

\begin{proof}
We associate to a decorated permutation $(w, \Delta)$ the $B$-orbit ${{\mathcal O}}_{(w,\Delta)}.$  We can rewrite a $B$-orbit $B\cdot (\F, [v])$ as $B \cdot (w(\fb), [y])$, $[y]\in\mathbb{P}^{n-1}$.  Recall that $B \cap wBw^{-1}=\mbox{Stab}_{B} (w(\fb))$, $H\subset B \cap wBw^{-1}$, and its Lie algebra $\fb\cap w(\fb) = \bigoplus_{i<j,w^{-1}(i)<w^{-1}(j)}\C E_{i,j}.$   It is then an easy computation to prove that there is $g\in B\cap wBw^{-1}$  so that $g\cdot y=v_{\Delta}$, where $\Delta\subset\{1,\dots,n\}$ is a decreasing sequence for $w^{-1}$ as above, so that $B\cdot (\F, [v])= {{\mathcal O}}_{(w,\Delta)}.$   This gives a map $\Psi$ from $B$-orbits on $\B_{n} \times\mathbb{P}^{n-1}$ to $\widehat{\Sigma}_{n}.$   It is routine to check that the map $(w,\Delta) \mapsto {{\mathcal O}}_{(w,\Delta)}$ is inverse to $\Psi.$
\end{proof}



\begin{nota} \label{n:marked}  The orbit ${{\mathcal O}}_{(w,\Delta)}\in B\backslash(\B_{n}\times\mathcal{O}_{i})$ if and only if $j_{k}=i$.  In this case, we refer to the decorated permutation $(w,\Delta)$ as an $i$-decorated permutation.  We denote the set of $i$-decorated permutations by $\widehat{\Sigma}_{n}(i)$.  For $(w,\Delta)\in \widehat{\Sigma}_{n}(i)$, we denote the  $S_{i}$-orbit on $\B_{n}$ corresponding to the $B$-orbit $\mathcal{O}_{(w,\Delta)}$ given by the rule in (\ref{eq:Sicorres}) by $Q_{(w,\Delta)}$.  Explicitly, this means

\begin{equation}\label{eq:markedcorres}
(w,\Delta)\leftrightarrow \calO_{(w,\Delta)}=B\cdot (w(\fb), [v_{\Delta}])\leftrightarrow Q_{(w,\Delta)}= S_{i} b\cdot w(\fb) \in S_{i}\backslash\B_{n},
\end{equation}
where $b\in B$ is such that $b\cdot [v_{\Delta}]=[e_{i}]$.

\end{nota}


\begin{proof}[Proof of Theorem \ref{thm:istd}]
Consider $(w,\Delta)\in\widehat{\Sigma}_{n}(i)$ with $\Delta=\{j_{1}<\dots<j_{k-1}<j_{k}=i\}$ and the corresponding $B$-orbit $\mathcal{O}_{(w,\Delta)}.$  Let $Q_{(w,\Delta)}$ be the corresponding $S_{i}$-orbit on $\B_{n}$ via (\ref{eq:markedcorres}).   Let $b\in B$ be the element of $G$ whose action on the standard basis of $\C^{n}$ is given by $b:e_{i}\mapsto e_{i}-\sum_{\ell=1}^{k-1} e_{j_{\ell}}$ and $b:e_{\ell}\mapsto e_{\ell}$ for all other $\ell$.   Then $b\cdot[v_{\Delta}]=[e_{i}]$, so that $Q_{(w,\Delta)}=S_{i}b\cdot w(\fb)$.  We show that $Q_{(w,\Delta)}$ contains a flag in $i$-standard form.  For this, let $i_{\ell}=w^{-1}(j_{\ell})$ for $\ell=1,\dots, k$.  Since $\Delta$ is a decreasing sequence for $w^{-1}$, the Borel subalgebra $w(\fb)$ stabilizes the flag $\tilde{\F}:=w(\mathcal{E})$ given by
 \begin{equation}\label{eq:firstflag} 
 \begin{split}
 \tilde{\F}=&(e_{w(1)}\subset\dots\subset e_{w(i_{k}-1)}\subset\underbrace{e_{i}}_{i_{k}}\subset e_{w(i_{k}+1)}\subset \dots\subset \underbrace{e_{j_{k-1}}}_{i_{k-1}}\subset
\dots\subset \underbrace{e_{j_{2}}}_{i_{2}}\subset \dots\subset\underbrace{e_{j_{1}}}_{i_{1}}\subset \\ &e_{w(i_{1}+1)}\subset\dots\subset e_{w(n)}).
\end{split}
\end{equation}
We claim that the flag $b\cdot\tilde{\F}$ is $S_{i}$-conjugate to the flag in $i$-standard form:
\begin{equation}\label{eq:secondflag}
\begin{split}
\F=&(e_{w(1)}\subset\dots\subset e_{w(i_{k}-1)}\subset\underbrace{ \he_{j_{k-1}}}_{i_{k}}\subset e_{w(i_{k}+1)}\subset \dots \subset \underbrace{\he_{j_{k-2}}}_{i_{k-1}}\subset \dots \subset \underbrace{\he_{j_{1}}}_{i_{2}}\subset \dots\subset \underbrace{e_{i}}_{i_{1}}\subset\\& e_{w(i_{1}+1)}\subset \dots\subset e_{w(n)}).
\end{split}
\end{equation}
Indeed, $b\cdot\tilde{\F}$ is the flag: 
\begin{equation*}
\begin{split}
b\cdot\tilde{\F}&=(e_{w(1)}\subset\dots\subset e_{w(i_{k}-1)}\subset\underbrace{ e_{i}-\sum_{\ell=1}^{k-1} e_{j_{\ell}}}_{i_{k}}\subset e_{w(i_{k}+1)}\subset \dots \subset \underbrace{e_{j_{k-1}}}_{i_{k-1}}\subset\\
& \dots \subset \underbrace{e_{j_{2}}}_{i_{2}}\subset \dots\subset  \underbrace{e_{j_{1}}}_{i_{1}}\subset e_{w(i_{1}+1)}\subset \dots\subset e_{w(n)}).
\end{split}
\end{equation*}
Now consider the element $x_{k-1}\in G$ whose action on the standard basis is given by:
$x_{k-1}: e_{j_{k-1}}\mapsto e_{j_{k-1}}-\sum_{\ell=1}^{k-2} e_{j_{\ell}}$ and $x_{k-1}: e_{m}\mapsto e_{m}$ for all other $m$.  Then $x_{k-1}\in S_{i}$ and 
\begin{equation*}
\begin{split}
x_{k-1}b\cdot\tilde{\F}
&=(e_{w(1)}\subset\dots\subset e_{w(i_{k}-1)}\subset\underbrace{ e_{i}-e_{j_{k-1}}}_{i_{k}}\subset e_{w(i_{k}+1)}\subset \dots \subset \underbrace{e_{j_{k-1}}-\sum_{\ell=1}^{k-2} e_{j_{\ell}}}_{i_{k-1}}\subset\\& \dots \subset \underbrace{e_{j_{2}}}_{i_{2}}\subset \dots\subset \underbrace{e_{j_{1}}}_{i_{1}}\subset e_{w(i_{1}+1)}\subset \dots\subset e_{w(n)}).
\end{split}
\end{equation*}
Now let $x_{k-2}\in G$ be the element which maps $e_{j_{k-2}}\to e_{j_{k-2}}-\sum_{\ell=1}^{k-3} e_{j_{\ell}}$ and fixes all other standard basis vectors.
Then $x_{k-2}\in S_{i}$ and 
\begin{equation*}
\begin{split}
x_{k-2}x_{k-1}b\cdot\tilde{\F}
&=(e_{w(1)}\subset\dots\subset e_{w(i_{k}-1)}\subset\underbrace{ e_{i}-e_{j_{k-1}}}_{i_{k}}\subset e_{w(i_{k}+1)}\subset \dots \subset \underbrace{e_{j_{k-1}}-e_{j_{k-2}}}_{i_{k-1}}\subset\\
&\dots\subset\underbrace{e_{j_{k-2}}-\sum_{\ell=1}^{k-3} e_{j_{\ell}}}_{i_{k-2}}\subset  \dots \subset \underbrace{e_{j_{2}}}_{i_{2}}\subset \dots\subset \underbrace{e_{j_{1}}}_{i_{1}}\subset  e_{w(i_{1}+1)}\subset \dots\subset e_{w(n)}).
\end{split}
\end{equation*}
Continuing in this fashion, we deduce that $b\cdot{\tilde{\F}}$ is $S_{i}$-conjugate to the flag 
\begin{equation*}
\begin{split}
x_{2}\dots x_{k-2}x_{k-1}b\cdot\tilde{\F}&=(e_{w(1)}\subset\dots\subset e_{w(i_{k}-1)}\subset\underbrace{ e_{i}-e_{j_{k-1}}}_{i_{k}}\subset e_{w(i_{k}+1)}\subset \dots \subset \underbrace{e_{j_{k-1}}-e_{j_{k-2}}}_{i_{k-1}}\subset\\
&\dots\subset\underbrace{e_{j_{k-2}}-e_{j_{k-3}}}_{i_{k-2}}\subset  \dots \subset \underbrace{e_{j_{2}}-e_{j_{1}}}_{i_{2}}\subset \dots\subset \underbrace{e_{j_{1}}}_{i_{1}}
\subset e_{w(i_{1}+1)}\subset \dots\subset e_{w(n)}).
\end{split}
\end{equation*}
Since this flag is $H$-conjugate to the flag $\F$ in $i$-standard form in Equation (\ref{eq:secondflag}), the $S_{i}$-orbit $Q_{(w,\Delta)}$ contains the flag $\F$ in $i$-standard form.


To prove uniqueness, consider an $S_{i}$-orbit $Q$ which contains a flag $\F$ in $i$-standard form.  
Without loss of generality, we may assume that $\F$ is given by 
\begin{equation}\label{eq:helperflag}
\F=(v_{1}\subset\dots\subset v_{i_{k}-1}\subset\underbrace{ \he_{j_{k-1}}}_{i_{k}}\subset v_{i_{k}+1}\subset \dots \subset \underbrace{\he_{j_{k-2}}}_{i_{k-1}}\subset \dots \subset \underbrace{\he_{j_{1}}}_{i_{2}}\subset \dots\subset \underbrace{e_{i}}_{i_{1}}\subset v_{i_{1}+1}\subset \dots\subset v_{n}),
\end{equation}
where the vectors $v_{m}$ are standard basis vectors.  
 We first show that the flag in (\ref{eq:helperflag}) is $B$-conjugate to the flag 
\begin{equation}\label{eq:wflag}
\tilde{\F}=(v_1 \subset\dots\subset v_{i_{k-1}}\subset\underbrace{e_{i}}_{i_{k}}\subset v_{i_{k+1}}\subset \dots\subset \underbrace{e_{j_{k-1}}}_{i_{k-1}}\subset\dots\subset \underbrace{e_{j_{2}}}_{i_{2}}\subset \dots\subset\underbrace{e_{j_{1}}}_{i_{1}}\subset v_{i_{1}+1}\subset\dots\subset v_{n}).
\end{equation}
  Consider the element $g\in G$ whose action on the standard basis of $\C^{n}$ is given by 
\begin{equation*}
g: e_{i}\mapsto \he_{j_{k-1}},\; e_{j_{m}}\mapsto -e_{j_{m}}+e_{j_{m-1}} \mbox{ for } m=2,\dots, k-1,\mbox{ and } g: e_{\ell}\mapsto e_{\ell} \mbox{ for all other } \ell.
\end{equation*}
Then $g\in B$ and computation shows that
\begin{equation*}
\begin{split}
g\cdot \tilde{\F}&=(v_{1}\subset\dots\subset \underbrace{\he_{j_{k-1}}}_{i_{k}}\subset\dots \subset \underbrace{-e_{j_{k-1}}+e_{j_{k-2}}}_{i_{k-1}}\subset\dots \subset \underbrace{-e_{j_{2}}+e_{j_{1}}}_{i_{2}}\subset \dots\subset \underbrace{e_{j_{1}}}_{i_{1}}\subset\dots \subset v_{n})\\
&=(v_{1}\subset\dots\subset \underbrace{\he_{j_{k-1}}}_{i_{k}}\subset\dots\subset \underbrace{\he_{j_{k-2}}}_{i_{k-1}}\subset\dots\ldots\ldots\subset\underbrace{\he_{j_{1}}}_{i_{2}}\subset\dots\ldots\ldots\subset \underbrace{e_{i}}_{i_{1}}\subset\dots\subset v_{n}),
\end{split}
\end{equation*}
which is the flag $\F$ in Equation (\ref{eq:helperflag}).  Thus, $B\cdot Q=B\cdot w_{\F}(\fb)$, where $w_{\F}\in W$ is such that $w_{\F}(\mathcal{E})=\tilde{\F}$. 
Since $i_k<i_{k-1}< \dots < i_1$ and $w_{\F}^{-1}(j_s)=i_s$ for $s=1, \dots, k$ and $j_k=i$, then $\Delta_{\F} = \{ j_1 < j_2 < \dots < j_{k-1} < j_k=i\}$  
 is a decreasing sequence for $w_{\F}^{-1}$ and the pair $(w_{\F},\Delta_{\F})\in\widehat{\Sigma}_{n}(i)$.   By Equation (\ref{eq:Sicorres}), the orbit $Q$ corresponds to the $B$-orbit 
$B\cdot (g\cdot \tilde{\F}, [e_{i}])=B\cdot (\tilde{\F}, g^{-1}\cdot[e_{i}])$.
Note that $g\cdot (e_{i}+e_{j_{k-1}}+\dots+e_{j_{2}}-e_{j_{1}})=e_i$, so that $g^{-1}\cdot e_i \in H\cdot [v_{\Delta_{\F}}].$  Hence, the orbit $Q=Q_{(w_{\F},\Delta_{\F})}$ by (\ref{eq:markedcorres}). 
Now if ${\mathcal{G}}$ is another flag in $Q$ in $i$-standard form, then 
$Q=Q_{(w_{\G},\Delta_{\G})}$, so $Q_{(w_{\F},\Delta_{\F})}=Q_{(w_{\G},\Delta_{\G})}$, and
thus $w_{\F}=w_{\G}$ and $\Delta_{\F}=\Delta_{\G}$ by Theorem \ref{thm:magyar}.  But $\Delta_{\F}=\Delta_{\G}$ and Equation (\ref{eq:wflag}) together  imply that $\F$ and $\G$ have the same vectors in all positions,  so that $\F=\G.$

\end{proof}

\begin{rem}\label{r:combinatorial}
 Theorem \ref{thm:istd} and its proof give a bijection:
$$
\widehat{\Sigma}_{n}(i)\leftrightarrow\{\mbox{Flags in $i$-standard form}\}.
$$
It is given in the forward direction by mapping $(w,\Delta)$ with $\Delta=\{j_{1}<\dots< j_{k-1}<i\}$ as above to the $i$-standard form $\F_{(w,\Delta)}$ given by placing hats on the vectors $e_{j_{k-1}},\dots, e_{j_{1}}$ in the flag $\tau_{\Delta} w(\mathcal{E}),$ where $\tau_{\Delta}$ is the $k$-cycle $(i,\, j_{k-1},\, \dots,\, j_{1}).$   The inverse bijection is given by mapping $\F$ to $(w_{\F},\Delta_{\F}).$  Moreover, this bijection is compatible with the orbit correspondence in Equation (\ref{eq:markedcorres}), i.e., $Q_{(w,\Delta)}=S_{i}\cdot \F_{(w,\Delta)}$, and conversely if $Q=S_{i}\cdot \F$ with $\F$ a flag in $i$-standard form then $Q=Q_{(w_{\F}, \Delta_{\F})}$. 
\end{rem}

\begin{rem}\label{p:std}
Note that using Magyar's parametrization makes the proof of Theorem \ref{thm:istd}  far simpler than the proof of the analogous result for $B_{n-1}\backslash\B_{n}$ (cf. Theorem 4.7 of \cite{CE21II}), which used a more involved argument.
\end{rem}



\begin{nota} \label{note:iandtildeflags}
For a flag $\F$ in $i$-standard form, we denote by $\tilde{\F}$ the unique $H$-stable flag in the $B$-orbit $B\cdot \F$ and by $\F^{i}$ the unique $H$-stable flag in the $B^{i}$-orbit $B^{i}\cdot \F$.  
\end{nota}
\noindent The following result generalizes Proposition 2.7 in \cite{Shpairs}.
\begin{prop}\label{p:iandtildeflags}
Let $\F\subset\B_{n}$ be a flag in $i$-standard form with 
$$
\F=(v_{1}\subset v_{2}\subset \dots \subset v_{k}\subset \dots v_{n}).
$$
If $\F$ contains no hat vectors then $\F=\tilde{\F}=\F^{i}$.  On the other hand, suppose that $\F$ is in the form of Equation (\ref{eq:helperflag}).
Then $\tilde{\F}$ is given by Equation (\ref{eq:wflag})
and the flag $\F^{i}$ is given by:
\begin{equation}\label{eq:i-flag}
\F^{i}=(v_{1}\subset\dots\subset v_{i_{k}-1}\subset\underbrace{ e_{j_{k-1}}}_{i_{k}}\subset v_{i_{k}+1}\subset \dots \subset \underbrace{e_{j_{k-2}}}_{i_{k-1}}\subset \dots \subset \underbrace{e_{j_{1}}}_{i_{2}}\subset \dots\subset \underbrace{e_{i}}_{i_{1}}\subset v_{i_{1}+1}\subset \dots\subset v_{n}).
\end{equation}
\end{prop}
\begin{proof}
If $\F$ contains no hat vectors, then $\F$ is $H$-stable, whence $\F=\tilde{\F}=\F^{i}$.  
On the other hand, if $\F$ is given by (\ref{eq:helperflag}), then we saw in the proof of Theorem \ref{thm:istd} that $\F$ is $B$-conjugate to the flag $\tilde{\F}$ of Equation (\ref{eq:wflag}).  It follows easily from the definition of the Borel subgroup $B^{i}$ as the stabilizer of the flag in (\ref{eq:Fiflag}) that $\F$ is $B^{i}$-conjugate to the flag in Equation (\ref{eq:i-flag}).
\end{proof}
\begin{proof}[Proof of Theorem \ref{thm:ShpairsforSi}]
By Theorem \ref{thm:istd}, each $S_i$-orbit $Q$ in $\B_{n}$ is of the form $S_{i}\cdot\F$ with $\F$ a flag in 
$i$-standard form.  Now 
$Q=S_{i}\cdot \F\subset B\cdot\tilde{\F}\cap B^{i}\cdot \F^{i}$, where 
$\tilde{\F}$ and $\F^{i}$ are given in Equations (\ref{eq:wflag}) and (\ref{eq:i-flag}). So to prove the 
Theorem, it suffices to show that the orbits $B\cdot \tilde{\F}$ and $B^{i}\cdot \F^{i}$ contain no common orbit besides $Q$.   By Remark \ref{r:combinatorial}, $Q=Q_{(w,\Delta)}$, where $w\in W$ is determined by $\tilde{\F}=w(\mathcal{E})$ and $\Delta\subset \{1,\dots, i\}$ consists of the indices of the hat vectors in $\F$ along with $i$.  Then it follows from Proposition \ref{p:iandtildeflags} that $\F^{i}=\tau_{\Delta}(\tilde{\F})$ where $\tau_{\Delta}$ is the $k$-cycle $(i,\, j_{k-1},\dots, j_{1}).$  If $Q^{\prime}=Q_{(w^{\prime}, \Delta^{\prime})}=S_{i}\cdot\mathcal{G}$ is another $S_{i}$-orbit, then similarly $\tilde{\mathcal{G}}=w^{\prime}(\mathcal{E})$ and
$\mathcal{G}^{i}=\tau_{\Delta^{\prime}}(\tilde{\mathcal{G}}).$   Hence, if $Q^{\prime}\subset  B\cdot\tilde{\F}\cap B^{i}\cdot \F^{i}$, then $\tilde{\F}=\tilde{\mathcal{G}}$ and $\F^{i}=\mathcal{G}^{i}$ so that $w=w^{\prime}$ and $\tau_{\Delta}=\tau_{\Delta^{\prime}}$.  Thus, $(w,\Delta)=(w^{\prime}, \Delta^{\prime})$ so that $Q=Q^{\prime}$.  


\end{proof}


\begin{dfn}\label{d:geoShpair}
Let $Q_B$ be a $B$-orbit and let $Q_{B^i}$ be a $B^{i}$-orbit in $\B_n$.  We call the pair $(Q_B,Q_{B^{i}})$ a {\it geometric $i$-Shareshian pair} if $Q_B \cap Q_{B^{i}}$ is nonempty, in which case it is a single $S_{i}$-orbit by Equation (\ref{eq:intersect}).
\end{dfn}

\begin{rem}\label{r:geoshar}
The map $\Sh_{i}$ from $S_{i}\backslash \B_n$ to geometric $i$-Shareshian pairs given by $Q\mapsto (B\cdot Q, B^{i}\cdot Q)$ is bijective with the inverse given by taking the intersection of the given $B$ and $B^{i}$-orbits.  This is a restatement of Corollary \ref{c:intersect}.
\end{rem}

We call the map $\Sh_{i}$ the {\it $i$-Shareshian map}, and we can describe it combinatorially as follows.   By the Bruhat decomposition, we can write any geometric $i$-Shareshian pair as $(B\cdot w(\mathcal{E}),B^{i}\cdot u^{i}(\mathcal{E}^{i}))$ for unique Weyl group elements $w,\, u^{i}\in W$.
\begin{dfn}\label{d:SHmap}
The $i$-Shareshian map is given by:
\begin{equation}\label{eq:Shmap}
\Sh_{i}: S_{i}\backslash\B_{n}\to W\times W; \;\Sh_{i}(S_{i}\cdot\F)=(w,u^{i}) \mbox{ where } \tilde{\F}=w(\mathcal{E})
\mbox{ and }\F^{i}=u^{i}(\mathcal{E}^{i})
\end{equation}
are the flags given in Notation \ref{note:iandtildeflags}, and $\F$ is a flag in $i$-standard form.
\end{dfn}

\begin{nota}\label{d:Shpair}
 We refer to a pair of Weyl group 
elements $(w,y)\in W\times W$ such that $(w,y)=\Sh_{i}(Q)$ for some $Q\in S_{i}\backslash\B_{n}$,  as an \emph{$i$-Shareshian pair}.  We denote the subset of $W\times W$ consisting of all $i$-Shareshian pairs as $\Sp_{i}\subset W\times W$.  
\end{nota}
\noindent The following statement is immediate from Remark \ref{r:geoshar}.
\begin{cor}\label{c:Shpair}
The $i$-Shareshian map $\Sh_{i}:S_{i}\backslash \B_n \to \Sp_{i}$ is bijective.
\end{cor}

\begin{rem}\label{r:starandtildeflags}
Let $Q=S_{i}\cdot \F\in\Borbitspace$ with $\F$ a flag in $i$-standard form and with $\Sh(Q)=(w, u^{i})$.  Then it follows from definitions that $B\cdot Q=B\cdot w(\mathcal{E})$ and $B^{i}\cdot Q=B^{i}\cdot u^{i}(\mathcal{E}^{i})$. Thus, $(w,u^{i})$ is an $i$-Shareshian pair if and only if $B\cdot w(\mathcal{E})\cap B^{i}\cdot u^{i}(\mathcal{E}^{i})\neq\emptyset$.


\end{rem}

We now describe necessary and sufficient conditions for a pair of Weyl group 
elements $(w,y)\in W\times W$ to be an $i$-Shareshian pair.  


\begin{prop}\label{p:Shpairs}
Let $\Delta=\{j_{1}<j_{2}<\dots<j_{k-1}<i\}$ be a subsequence of $\{1,\dots,i\}$ 
containing $i$ and let $\tau_{\Delta}\in \Sigma_{n}$ be the $k$-cycle
$\tau_{\Delta}:=(i,\, j_{k-1}, \dots, j_{2}, j_{1})$.  Define $\sigma_{i}\in\Sigma_{n}$ to be 
the $i$-cycle $\sigma_{i}=(i,i-1,\dots, 2,1)$.  Then a pair of Weyl group elements 
$(w, y)\in W\times W$ is an $i$-Shareshian pair if and only if the following two conditions hold: 
\begin{enumerate}
\item $y=\tau_{\Delta}w\sigma_{i}^{-1}$ for some subsequence $\Delta$ of $\{1,\dots, i\}$ containing $i$ as above.
\item $\Delta$ is a decreasing sequence for $w^{-1}$, i.e., $w^{-1}(i)<w^{-1}(j_{k-1})<\dots <w^{-1}(j_{1})$.
\end{enumerate}
\end{prop}

\begin{proof}
Suppose $Q\in S_{i}\backslash\B_{n}$ with $Q=Q_{(w,\Delta)}=S_{i}\cdot \F$ 
with $\Delta=\{j_{1}<\dots< j_{k-1}<i\}$ and $\F$ the unique flag in $i$-standard form in $Q$.  Then $\tilde{\F}=w(\mathcal{E})$ and $\F^{i}=\tau_{\Delta}(\tilde{\F})$ by Remark \ref{r:combinatorial} and
Proposition \ref{p:iandtildeflags} (if $\F$ has no hat vectors, then $\Delta=\{i\}$ and $\tau_{\Delta}=\{id\}$).  Recall the flag $\mathcal{E}^{i}$ of Equation (\ref{eq:Fiflag}) and note that $\sigma_{i}^{-1}(\mathcal{E}^{i})=\mathcal{E}$.  It then follows that 
$$
\F^{i}=\tau_{\Delta}(\tilde{\F})=\tau_{\Delta}w (\mathcal{E})=\tau_{\Delta}w\sigma_{i}^{-1}(\mathcal{E}^{i}),
$$
whence $\Sh_{i}(Q)=(w,\tau_{\Delta}w\sigma_{i}^{-1})$.  

Conversely, suppose we are given $(w,y)\in W\times W$ satisfying (1) and (2).  Then $(w,\Delta)\in\widehat{\Sigma}_{n}(i)$.  Let $\F$ be the unique flag in $i$-standard form in the orbit $Q_{(w,\Delta)}$.  Then by Remark \ref{r:combinatorial} and Proposition \ref{p:iandtildeflags}, $\tilde{\F}=w(\mathcal{E})$ and $\mathcal{F}^{i}=\tau_{\Delta}(\tilde{\F})$, so that $\Sh_{i}(Q)=(w,y)$.


\end{proof}

\begin{rem}\label{r:whatisdelta}
Let $\Sh_{i}(Q)=(w, u^{i})$ be the $i$-Shareshian pair for an orbit $Q=S_{i}\cdot \F$ with $\F$ in $i$-standard form.  It follows from Proposition \ref{p:Shpairs} and its proof that  
$\tau_{\Delta}:=u^{i}\sigma_{i} w^{-1}$ is a cycle $(i, j_{k-1},\dots, j_{1})$, and the hat vectors in $\F$ are $\he_{j_{k-1}},\dots, \he_{j_{1}}$.  Further, the pair $(w,\Delta)$ with 
$\Delta:=\{j_{1}<j_{2}\dots<j_{k-1}<i\}$ is an $i$-decorated permutation, and $Q=Q_{(w,\Delta)}$ with $Q_{(w,\Delta)}$ as in Notation \ref{n:marked}.  Conversely, if 
$Q=Q_{(w,\Delta)}$ with $(w,\Delta)\in \widehat{\Sigma}_{n}(i)$ then $\Sh_{i}(Q)=(w,\tau_{\Delta} w\sigma_{i}^{-1})$.  
\end{rem}

It follows from Remarks \ref{r:combinatorial} and \ref{r:whatisdelta} that we have now established identifications between three different parameterizations of $S_{i}$-orbits on $\B_{n}$:
\begin{equation}\label{eq:differentparams}
\{\mbox{Flags in $i$-standard form}\}\leftrightarrow \widehat{\Sigma}_{n}(i)\leftrightarrow \{ i\mbox{-Shareshian pairs}\}.
\end{equation}
At times, it will be convenient for us to use a slightly different version of an $i$-Shareshian pair.
\begin{dfn}\label{d:standardizedpair}
Recall the $i$-cycle $\sigma_{i}=(i,\, i-1,\dots, 1)\in\Sigma_{n}$ defined in the statement of Proposition \ref{p:Shpairs}. 
Let $Q\in S_{i}\backslash\B_{n}$.  If $\Sh_{i}(Q)=(w,u^{i})$, we define the \emph{$i$-standardized Shareshian pair} for $Q$ to be $\widetilde{\Sh}_{i}(Q)=(w,u)$ where 
$u:=\sigma_{i}^{-1}u^{i}\sigma_{i}$.

\end{dfn}

\section{Combinatorial Models for $S_{i}$-orbits: Lists and PILs}\label{s:models}

In this section, we give an alternative combinatorial 
model for $S_{i}\backslash\B_{n}$ to decorated permutations using the notion of partitions into lists.   We compute exponential generating functions for 
the sequence $\{|S_{n-i}\backslash\B_{n}|\}_{n\geq i}$ and use the result to compute the exponential generating function for the sequence $\{|G_{\Delta}\backslash (\B_{n}\times\B_{n}\times\mathbb{P}^{n-1})|\}_{n\geq 1}$. Finally, we give a third combinatorial model for decorated permutations using partial permutations, which is suggested from the study of the latter generating function. 

  We begin by recalling the parameterization of $\Borbitspace$ in Section 3 of \cite{CE21II}. 

\begin{dfn}\label{dfn:lists}
\begin{enumerate}
\item We define a \emph{list} on a subset $\{a_{1},\dots, a_{n}\}$ of the non-negative integers to be a $k$-tuple $\sigma=(a_{i_{1}},\dots, a_{i_{k}}) \in \mathbb{Z}_{\geq 0}^k$, where $i_{s}\neq i_{t}$ for all $s\neq t$.  
For the list $\sigma$, we say
$\supp(\sigma)=\{ a_{i_1}, \dots, a_{i_k} \}$, the length $\ell(\sigma)=k$,
and $\sigma(j)=a_{i_j}$ for $1 \le j \le k$.  We denote the set of lists of $\{1,\dots, n\}$ of length $k$ by $Li(n,k)$. 
\item  Let $A\subset \mathbb{Z}_{\geq 0}$ be any finite subset.  A $PIL$ \emph {(partition into lists)} is a set of lists 
$\Sigma=\{\sigma_{1},\, \dots, \, \sigma_{k}\}$ on $A$ such that $\{\supp(\sigma_{1}),\dots, \supp(\sigma_{k})\}$ is a partition of $A$. 
\item For a finite subset $A\subset\mathbb{Z}_{\geq 0}$, we denote the set consisting of all PILs of $A$ by $PIL(A)$.  
In the case where $A=\{1,\dots, n\}$, we denote $PIL(A)$ by $PIL(n)$. 
\end{enumerate}
\end{dfn}

\begin{rem}\label{r:stdflagpil}
 Proposition 4.5 from  \cite{CE21II} asserts that there is a bijection from flags in standard form to $PIL(n)$ given as follows.  Let $\F$ be a flag in standard form with $\F = (v_1 \subset \cdots \subset v_n).$   Let $1 \le k_1 < \dots < k_r < n$ be the subsequence of $\{ 1, \dots, n \}$ such that $\{ v_{k_1}, \dots, v_{k_r} \}$ are the hat vectors in $\F$, and let $v_{k_{\ell}}={\hat{e}}_{s_{k_{\ell}}}$ for $\ell = 1, \dots, r$ and let $v_t=e_{s_t}$ for $t\not\in \{ 1, \dots, r \}.$
Then the PIL associated to $\F$ is
\[
\{(s_1,\dots, s_{k_1 - 1},s_{k_1}),(s_{k_1+1},\dots, s_{k_2}),\dots,(s_{k_{r-1}+1}, \dots, s_{k_r}),(s_{k_{r}+1}, \dots, s_n)\}.
\]
We note that the formula for this map was stated incorrectly in loc. cit.
\end{rem}


\noindent To parameterize $S_{i}\backslash\B_{n}$, we introduce the set:
\begin{equation}\label{eq:PILwithlist}
PIL(n,n-i)\subset Li(n,n-i)\times PIL(n)\mbox{ defined by } PIL(n,n-i):=\{(\sigma, \Sigma): \;\sigma\in\Sigma\}. 
\end{equation}
We note that the symmetric group acts on $PIL(n)$ through its action on indices in lists.
\begin{exam}\label{ex:PIL31}
The set $PIL(3,1)$ consists of $9$ elements: 

$((1), \{(1),(2),(3)\}),\,$ $((2), \{(1),(2),(3)\}),\, $ $((3), \{(1),(2),(3)\}),\,$ $((1),\{(1), (2,3)\}),$  

$((1),\{(1), (3,2)\}),\,((2),\{(2), (1,3)\}),\, ((2),\{(2), (3,1)\}),\, ((3), \{(3), (12)\}), ((3), \{ (3), (2,1)\}).$ 
\end{exam}

\begin{prop}\label{p:countSi}
The number of $S_{i}$-orbits on $\B_{n}$ is 
\begin{equation}\label{eq:countSi}
|S_{i}\backslash\B_{n}|=|PIL(n,n-i)|. 
\end{equation}
\end{prop}

\begin{rem}
When $i=n$, Equation (\ref{eq:countSi}) reduces to the statement 
that $|\Borbitspace|=|PIL(n)|$, which is Theorem 3.3 of \cite{CE21II}.

\end{rem}

\begin{proof}
By Theorem \ref{thm:istd}, it suffices to construct a bijection
 $\Phi$ from the set $\{\F\in\B_{n}:\, \F \mbox{ is a flag in $i$-standard form}\}$ to the set $PIL(n,n-i)$.  
Let $\F=(v_{1}\subset\dots\subset v_{n})$ be a flag in $i$-standard form.  It follows from Definition \ref{d:istd} and Remark \ref{r.glstandard} that the vectors $e_{i+1},\dots, e_{n}$ occur in $\F$.  Suppose that $v_{j_{i+1}}=e_{i+1},\dots, v_{j_{n}}=e_{n}$.  Let $L$ be the standard Levi subgroup of $G$ corresponding to the  subset $\{\alpha_{1},\dots,\alpha_{i-1}\}$ of standard simple roots, and let $W_{L}$ be its Weyl group.   Let ${}^LW := \{ w\in W : w^{-1}(\Delta_L) \subset \Phi^+ \}$ and recall that each coset in $W_{L}\backslash W$ has a unique minimal length representative in ${}^LW$ (see 5.13 of \cite{Kostant-BW}).
It follows that there is a unique $w\in {}^LW$ such that $w(j_{i+1}) = i+1, w(j_{i+2}) = i+2, \dots, \, w(j_{n})=n.$  Further, if $j_1 < \dots < j_i$ are the remaining indices, $w(j_{\ell})=\ell$ for $\ell = 1, \dots, i.$ 
If we let $u_{\ell}=v_{j_{\ell}}$ for $\ell=1, \dots, i,$ we obtain a flag
$\F^{\prime}:=(u_{1}\subset \dots \subset u_{i})$ which by Definition \ref{d:istd} and Remark \ref{r:usual} is a flag in standard form in the flag variety $\B_{i}$ of $\fgl(i)$. By Remark \ref{r:stdflagpil}, the standard flag $\F^{\prime}$ corresponds to a unique PIL $\Sigma^{\prime}\in PIL(i)$.  We therefore declare:
\begin{equation}\label{eq:Phidefn}
\Phi(\F):=(w^{-1}\cdot (i+1,\dots, n), \{w^{-1}\cdot (i+1,\dots, n), w^{-1}\cdot\Sigma^{\prime}\})\in PIL(n,n-i).
\end{equation}

To see that $\Phi$ is surjective, consider  $(\sigma, \Sigma)\in PIL(n,n-i)$.  Since $\ell(\sigma)=n-i$, $\sigma=(j_{i+1},\dots, j_{n})$ with $\mbox{supp}(\sigma)=\{j_{i+1},\dots, j_{n}\}\subset \{1,\dots, n\}.$  As above, there is a unique element  $w\in{}^{L}W$ such that $w(j_{\ell})=\ell$ for $\ell=i+1, \dots, n.$ Then $w(\sigma)=(i+1, \dots, n)$ is a list in $w\cdot \Sigma$, and we let $\Sigma^{\prime}\in PIL(i)$ denote $\Sigma$ with $w(\sigma)$ omitted.   Let $\F^{\prime}=(u_1 \subset \dots \subset u_i)$ be the flag in $\B_{i}$ in standard form corresponding to the PIL $\Sigma^{\prime}$ via Remark \ref{r:stdflagpil}.  Let  $\F=(v_{1}\subset v_{2}\subset\dots\subset v_{n})$ be the flag with $v_{j_{k}}=e_{k}$ for $k=i+1,\dots, n$ and $v_{j_k}=u_k$ for $k=1, \dots, i.$   It follows easily that $\Phi(\F)=(\sigma,\Sigma)$, so that $\Phi$ is surjective.  For injectivity, if  $\F$ and $\G$ are flags in $i$-standard form and $\Phi(\F)=\Phi(\G)$, then it follows from the construction that $\F$ and $\G$ define the same element of ${}^LW,$ and similarly that $\F$ and $\G$ define the same flag in $\B_{i}$ in standard form, and then it follows easily  that $\F$ coincides with $\G$.

\end{proof}

  Recall that if $\{ f(n) \}_{n\in \mathbb{Z}_{\ge 0}}$ is a sequence of complex numbers, the corresponding exponential generating function $E_f(x)$ is the formal power series 
  $$\sum_{n=0}^{\infty} \frac{f(n) }{n!} x^n.$$  
 For functions $f, g:{\mathbb{Z}}_{\ge 0} \to \C$, define $s=s(f,g, \cdot):{\mathbb{Z}}_{\ge 0} \to \C$ by 
 $$s(n)=\sum_{k=0}^{n}{n\choose k} f(k)g(n-k).$$ 
 Then the multiplication principle for exponential generating functions asserts that $E_s(x)=E_f(x)\cdot E_g(x)$ (see Proposition 5.1.3 and its proof in \cite{Stanley99}).  Using Proposition \ref{p:countSi} and the multiplication principle for exponential generating functions, we can compute an exponential generating function for the sequence $\{|S_{n-i}\backslash \B_{n}|\}_{n\geq i}$, where we adopt the convention that $S_{0}=B$.    
  
  
\begin{prop}\label{p:egfforSi}
Fix an $i\in\mathbb{Z}_{\geq 0}$.  
The exponential generating function for the sequence 
\begin{equation}\label{eq:egfforSi}
\{|S_{n-i}\backslash \B_{n}|\}_{n\geq i} \mbox{ is } x^{i}e^{\frac{x}{1-x}}.
\end{equation}
\end{prop}
\begin{proof}
By Proposition \ref{p:countSi}, it suffices to show that the exponential generating function for the sequence 
$\{|PIL(n,i)|\}_{n\geq i}$ is $x^{i}e^{\frac{x}{1-x}}$.  
It follows from Theorem 3.3 and Proposition 3.12 of \cite{CE21II} that $e^{\frac{x}{1-x}}=E_{g}(x)$, where $g(n)$ is the sequence $g(n)=|PIL(n)|$ for $n\geq 0$.  Note that $x^{i}=E_{f}(x)$, where $f(k)$ is the sequence 
$f(k)=\delta_{k, i} i !$.  It follows from the multiplication principle for exponential generating functions discussed above that $x^{i}e^{\frac{x}{1-x}}=E_{s}(x)$ where $s(n)$ is the sequence with 
$s(n)={n\choose i} i! |PIL(n-i)|$.  The Proposition then follows from the claim that 
\begin{equation}\label{eq:countformula}
|PIL(n,i)|={n\choose i} i! |PIL(n-i)|.
\end{equation}
For this, elements of $PIL(n,i)$ correspond to pairs $(\sigma, \Sigma^{\prime})$ where $\sigma$ is an ordered subset of cardinality $i$ and once $\sigma$ is chosen,  $\Sigma^{\prime}$ is a PIL on the set of $n-i$ elements not in $\sigma$.    There are $\frac{n!}{(n-i)!}$ ways to choose $\sigma$, and the claim follows.
\end{proof}


We can now compute the exponential generating function for the sequence 
$\{|G_{\Delta}\backslash(\B_{n}\times\B_{n}\times\mathbb{P}^{n-1})|\}_{n\geq 1}.$  
\begin{thm}\label{thm:bigegf}
The exponential generating function for the sequence 
\begin{equation}\label{eq:biggenfun}
\{|G_{\Delta}\backslash(\B_{n}\times\B_{n}\times\mathbb{P}^{n-1})|\}_{n\geq 1}\mbox{ is } \frac{e^{\frac{x}{1-x}}-x}{1-x}.
\end{equation}
\end{thm}
\begin{proof}
Let $a_n:=|G_{\Delta}\backslash(\B_{n}\times\B_{n}\times\mathbb{P}^{n-1})|.$   Note that Equation (\ref{eq:Gdeltacorres}) implies that $a_n=|B\backslash(\B_{n}\times\mathbb{P}^{n-1})|.$ It follows from Equations (\ref{eq:orbitspacedecomp}) and (\ref{eq:Sicorres}) that
\begin{equation}\label{eq:bigcount}
a_n=\displaystyle\sum_{k=0}^{n-1} |S_{n-k}\backslash\B_{n}|.
\end{equation}
By  Proposition \ref{p:countSi} along with Equations (\ref{eq:countformula}) and (\ref{eq:bigcount}), we deduce that
\begin{equation}\label{eq:nextcount}
a_n=\displaystyle\sum_{k=0}^{n-1} {n\choose k} k! |PIL(n-k)| \mbox{ for } n\geq 1.
\end{equation}
It follows that if $a(x)= \sum_{n=0}^{\infty} \frac{a_n}{n!}x^n$ is the exponential generating function for $a_n$, then $a(x)=f(x):=\frac{e^{\frac{x}{1-x}}-x}{1-x}$ follows from the assertion that
\begin{equation}\label{eq:nthderiv}
f^{(n)}(0)=\displaystyle\sum_{k=0}^{n-1} {n\choose k} k! |PIL(n-k)| \mbox{ for } n\geq 1.
\end{equation}
We compute $f^{(n)}(0)$ using the generalized Leibniz formula 
\begin{equation}\label{eq:genLiebniz}
(gh)^{(n)}(0)=\sum_{k=0}^{n} {n\choose k} g^{(n-k)} (0) h^{(k)}(0),
\end{equation}
with $g(x)=e^{\frac{x}{1-x}}-x$ and $h(x)=\frac{1}{1-x}.$
By Proposition 3.12 in \cite{CE21II}, we know
$$
g(x)=\displaystyle\left(1+\displaystyle\sum_{k=2}^{\infty} \frac{|PIL(k)|}{k!} x^{k}\right).
$$
Using Equation (\ref{eq:genLiebniz}) with $g(x)=1+\displaystyle\sum_{k=2}^{\infty} \frac{|PIL(k)|}{k!} x^{k}$ and $h(x)=\sum_{k=0}^{\infty}x^{k}$, we obtain 
\begin{equation}\label{eq:last}
f^{(n)}(0)=n! + \displaystyle\sum_{k=0}^{n-2} {n\choose k} |PIL(n-k)| k!.
\end{equation}
Since ${n\choose n-1} (n-1) ! |PIL(1)|=n!$, Equation (\ref{eq:last}) coincides with the right hand side of Equation (\ref{eq:nthderiv}) and the proof is complete.

\end{proof}

The following result makes computation of $|S_i\backslash\B_{n}|$ straightforward.

\begin{prop}\label{p:nextone}
Let $n\in\mathbb{N}$ and fix an index $i$ with $1\leq i\leq n$.  Then 
\begin{equation}\label{eq:nextone}
|S_{i}\backslash\B_{n+1}|=(n+1) |S_{i}\backslash \B_{n}|.
\end{equation}
\end{prop}
\begin{proof}
By Proposition \ref{p:countSi}, $|S_{i}\backslash \B_{n}|=|PIL(n,n-i)|$.  Replacing 
$i$ by $n-i$ in Equation (\ref{eq:countformula}), we obtain 
\begin{equation}\label{eq:reducecount}
|S_{i}\backslash \B_{n}|=|PIL(n,n-i)|=n(n-1)\dots (i+1) |PIL(i)|.
\end{equation}
Replacing $n$ with $n+1$ in the last formula implies
\begin{equation}\label{eq:secreducecount}
|S_{i}\backslash\B_{n+1}|=|PIL(n+1,n+1-i)|=(n+1)n(n-1)\dots (i+1) |PIL(i)|.
\end{equation}
Comparing Equations (\ref{eq:reducecount}) and (\ref{eq:secreducecount}) yields 
Equation (\ref{eq:nextone}).



\end{proof}

\begin{cor}\label{c:Sirecursion}
For any $1\leq i\leq n$, we have: 
\begin{equation}\label{eq:orbitcount}
|S_{i}\backslash\B_{n}|=n(n-1)\dots (i+1) |B_{i-1}\backslash \B_{i}|.
\end{equation}
\begin{proof}
By Theorem 3.3 of \cite{CE21II}, $|PIL(i)|=|B_{i-1}\backslash\B_{i}|$.  The result then follows from Equation (\ref{eq:reducecount}).

\end{proof}
\end{cor}

\begin{exam}\label{ex:formulatable}
Formulas for $|B_{n-1}\backslash \B_n|$ can be derived from Lah numbers (see Proposition 3.11 of \cite{CE21II}), and these enable us to compute $|S_i\backslash \B_n|$ using the preceding results.   It is useful to view the values of
 $\{|S_{i}\backslash \B_{n}|\}_{1\leq i\leq n,\, n\geq 1}$ in a triangle.  We give the first six rows of the triangle below.  The top row being $n=6$ and the bottom row $n=1$.  From left to right, the entries for row n are $|S_{1}\backslash\B_{n}|=n!,$ $|S_{2}\backslash\B_{n}|, \dots, |S_{n-1}\backslash\B_{n}|$ and finally $|B_{n-1}\backslash\B_{n}|=|PIL(n)|$. 
\begin{gather*}
720 \qquad 1080 \qquad1560 \qquad 2190 \qquad 3006 \qquad 4051  \\
        120 \qquad 180  \qquad 260   \qquad 365  \qquad 501   \\
        24\qquad 36 \qquad 52\qquad 73 \\
        6 \qquad 9 \qquad 13 \\
        2\qquad 3 \\
        1 
\end{gather*}
We note that the $i$-th northeast diagonal ($i\geq 1)$ represents the sequence $\{|S_{n-i+1}\backslash\B_{n}|\}_{n\geq i}$ whose exponential generating function is $x^{i-1} \exp^{\frac{x}{1-x}}$ by Proposition \ref{p:egfforSi}. The $i$-th northwest diagonal is the sequence $\{|S_{i}\backslash\B_{n}|\}_{n\geq 1}$.  The relation given by Equation (\ref{eq:nextone}) is clearly visible in these diagonals.  The row sums of the triangle yield the sequence $\{|B\backslash(\B_{n}\times\mathbb{P}^{n-1})|\}_{n\geq 1}$ 

\end{exam}

\begin{exam}\label{ex:firstsix}
Using the above triangle, we see that the first six values of the sequence $\{|B\backslash(\B_{n}\times \mathbb{P}^{n-1})|\}_{n\geq 1}$ are given by 
$1, \, 5,\, 28, \, 185, \, 1426, \, 12607$.  Comparison of this sequence with sequences in the Online Encyclopedia of Integer Sequences (OEIS) shows that it agrees with the sequence counting strictly partial permutations, labelled as $\{|\mathcal{I}_{n}|\}_{n\geq 1}$ (OEIS sequence A070779).  
\end{exam}

\begin{dfn}
A strict partial permutation is a bijection $\kappa: A\to B$ between proper subsets $A, \, B$ of $\{1,\dots, n\}$.  We denote the set of strict partial permutations of 
$\{1,\dots, n\}$ by $\mathcal{I}_{n}$.
\end{dfn}

 In fact, we can write down a bijection between $B\backslash(\B_{n}\times \mathbb{P}^{n-1})$ and the set 
of strictly partial permutations using decorated permutations.



\begin{prop}\label{p:strictly}
There is a bijection between the set $\widehat{\Sigma}_{n}$ of all decorated permutations and the set $\mathcal{I}_{n}$of all strictly partial permutations of $\{1,\dots, n\}$ given by the map
\begin{equation}\label{eq:partialbiject}
\Phi:(w,\Delta):\mapsto w|_{\{1,\dots, n\}\setminus w^{-1}(\Delta)} \mbox{ for } (w,\Delta)\in \widehat{\Sigma}_{n},
\end{equation}
 where $w|_{\{1,\dots, n\}\setminus w^{-1}(\Delta)}$ denotes the restriction of the permutation $w\in\Sigma_{n}$ to the subset $\{1,\dots, n\}\setminus w^{-1}(\Delta)$.
\end{prop}
\begin{proof}
We first show that $\Phi$ is surjective.  Suppose $\tau$ is a strictly partial permutation with 
$\tau(j_{k})=i_{k}$ for $k=1,\dots, r$ where $r<n$.  Let $\{m_{1}< m_{2}<\dots< m_{n-r}\}:=\{1,\dots, n\}\setminus \{j_{1},\dots, j_{r}\}$, and define 
$\Delta:=\{1,\dots, n\}\setminus\mbox{Im}(\tau)$, where $\mbox{Im}(\tau)$ denotes the image of the partial permutation $\tau$.   Let $\Delta=\{\ell_{1}<\dots< \ell_{n-r}\}$.  Now extend $\tau$ to a ``full" permutation via the assignment:  $w(m_{1})=\ell_{n-r},\dots, w(m_{n-r})=\ell_{1}$ and $w|_{\{j_{1},\dots, j_{r}\}}=\tau$.   By construction, $(w,\Delta)$ is a decorated permutation since
$w^{-1}(\ell_{n-r})=m_{1}<\dots<w^{-1}(\ell_{1})=m_{n-r}$, so that $\Delta$ is a decreasing sequence for $w^{-1}$.  Note that $\{1,\dots, n\}\setminus w^{-1}(\Delta)=\{j_{1},\dots, j_{r}\}$.  Thus, $\Phi(w,\Delta)=w|_{\{1,\dots, n\}\setminus w^{-1}(\Delta)}=\tau$, and $\Phi$ is surjective.  If also $(w_1,\Delta_1) \in \widehat{\Sigma}_n$ and  $\Phi(w_1,\Delta_1)=\tau$, then $\Delta_1=\{ 1, \dots, n\} \setminus \mbox{Im}(\tau) = \Delta.$  Thus, $w_1:\{m_1, \dots, m_{n-r} \} \to \{ \ell_1, \dots, \ell_{n-r} \},$ so since $(w_1, \Delta_1) \in \widehat{\Sigma}_n$, it follows that 
$w_1^{-1}(\ell_{n-r})=m_1 < \dots < w_1^{-1}(\ell_{1})=m_{n-r}.$   Thus, $w_1=w$, and $\Phi$ is injective.
\end{proof}

\begin{rem}
We note that Fresse and Nishiyama also note the same identification between decorated permutations and partial permutations in \cite{FN}, although their definition of each of these objects is slightly more general than ours.
\end{rem}

\section{Monoid actions on $S_{i}$-orbits}\label{s:monoid} 
In this section, we generalize an extended monoid action on $\Borbitspace$ discussed in Section 4 of \cite{CE21I} to an extended monoid action on $S_{i}\backslash\B_{n}$.  Our main result is that the $i$-Shareshian map of Equation (\ref{eq:Shmap}) intertwines this  extended monoid action with a variant of the classical monoid action of $W$ on itself extended to the product $W\times W$.  We end the section by presenting the Bruhat graph of $S_{2}\backslash\B_{3}$ along with all the monoid actions labelled.

\subsection{Background on Monoid actions}\label{ss:monoidbackground}

For more details on the subsequent material, we refer the reader to \cite{RS}, \cite{Vg}, \cite{CEexp}, \cite{CE21I}, and other sources.
Let $R$ be a connected reductive algebraic group with Lie algebra $\fr$, let $\B=\B_{R}$ be the flag 
variety of $R$, and let $M$ be an algebraic subgroup of $R$ acting on $\B$ with finitely many orbits.  Identify $\B \cong R/B_{R}$, for a Borel subgroup $B_{R}\subset R$ and let $\fb_{\fr}=\mbox{Lie}(B_{R})\subset\fr$.  
Let $\Pi_{\fr}$ be the set of simple roots defined by the Borel subalgebra $\fb_{\fr}$, and let $S_{R}$ be the simple reflections of the Weyl group $W$ of $R$ corresponding to $\Pi_{\fr}$.  For $\alpha\in \Pi_{\fr}$, let $\mathcal{P}_{\alpha}$ be the variety of all parabolic subalgebras of $\fr$ of type $\alpha$ and consider the $\mathbb{P}^{1}$-bundle $\pi_{\alpha}:\B \to {\mathcal{P}}_{\alpha}$.  
For $\alpha \in \Pi_{\fr}$ with corresponding reflection $s=s_{\alpha} \in W$, we define an operator $m(s)$ on the set of orbits $M\backslash \B$ following the above sources.  For $Q_M \in M\backslash \B$, let $m(s)*Q_{M}$ be the unique $M$-orbit open and dense in $\pi_{\alpha}^{-1}(\pi_{\alpha}(Q_{M})).$  
Since $\pi_{\alpha}: \pi_{\alpha}^{-1}(\pi_{\alpha}(Q_{M}))\to \pi_{\alpha}(Q_{M})$ is an $M$-equivariant 
$\mathbb{P}^{1}$-bundle, the orbit  $Q_{M} \not= m(s)*Q_{M}$ if and only if $\dim(m(s)*Q_{M})=\dim(Q_{M})+1.$

Computation of $m(s)*Q_{M}$ depends on the \emph{the type of the root} $\alpha$ for the orbit $Q_{M}$, which is determined as follows.  For $\fb^{\prime}\in Q_M$, let $\fp_{\alpha}^{\prime} = \pi_{\alpha}(\fb^{\prime})$ and let $B^{\prime}$ and $P_{\alpha}^{\prime}$ be the corresponding parabolic subgroups of $R$, and let $V_{\alpha}^{\prime}$ be the the solvable radical of $P_{\alpha}^{\prime}$.  Consider the group 
$S_{\alpha}^{\prime}:= P_{\alpha}^{\prime}/V_{\alpha}^{\prime}$ isogenous to $SL(2)$ and its subgroup 
$M_{\alpha, \fb^{\prime}}:= (M \cap P_{\alpha}^{\prime})/(M\cap V_{\alpha}^{\prime})$.
 Using the standard identification $\pi_{\alpha}^{-1}(\pi_{\alpha}(Q)) \cong M\times_{M\cap P_{\alpha}^{\prime}} P_{\alpha}^{\prime}/B^{\prime},$ we see that $M$-orbits in
 $\pi_{\alpha}^{-1}(\pi_{\alpha}(Q))$ correspond to $M_{\alpha, \fb^{\prime}}$-orbits in 
$P^{\prime}_{\alpha}/B^{\prime}\cong {\PR}^1.$  

\begin{dfn}\label{d:roottype}
\par\noindent (1) If $M_{\alpha, \fb^{\prime}}$ is solvable and contains the unipotent radical of a Borel subgroup of $S_{\alpha}^{\prime}$, then $\alpha$ is called a complex root for $Q_M$.  If $M_{\alpha, \fb^{\prime}}\cdot \fb^{\prime}=\fb^{\prime}$, then $\alpha$ is complex stable for $Q_M$ and otherwise $\alpha$ is complex unstable for $Q_M$.
\par\noindent (2) If $M_{\alpha, \fb^{\prime}}=S_{\alpha}^{\prime}$, then $\alpha$ is called a compact imaginary root for $Q_M.$
\par\noindent (3) Suppose $M_{\alpha, \fb^{\prime}}$ is one-dimensional and reductive.
If $M_{\alpha,\fb^{\prime}}\cdot \fb^{\prime}=\fb^{\prime}$, then $\alpha$ is called a noncompact imaginary root for $Q_M$, while if $M_{\alpha,\fb^{\prime}}\cdot \fb^{\prime}\not=\fb^{\prime},$ then $\alpha$ is called a real root for $Q_M.$
\end{dfn}

It is well-known and easy to prove that the type of the root depends only on the orbit $Q_M$ and not on the point $\fb^{\prime}.$ The following result is also well-known. 

\begin{prop}\label{p:stableandnc}
Let $Q\in M\backslash\B$ with $Q=M\cdot \Ad(v)\fb_{\fr}$ and $\alpha\in \Pi_{\fr}$.
\begin{enumerate}
\item If $\alpha$ is complex stable for $Q$, then $\pi_{\alpha}^{-1}(\pi_{\alpha}(Q))$ consists of two $M$-orbits:
$$
\pi_{\alpha}^{-1}(\pi_{\alpha}(Q))=Q\cup M\cdot \Ad(vs_{\alpha})\fb_{\fr},
$$
and $\ms*Q= M\cdot \Ad(vs_{\alpha})\fb_{\fr}$. 
\item If $\alpha$ is non-compact imaginary for $Q$ then $\pi_{\alpha}^{-1}(\pi_{\alpha}(Q))$ consists of either two or three $M$-orbits. The open orbit
$\ms*Q=M\cdot \Ad(vu_{\alpha})\fb_{\fr}$, where $u_{\alpha}\in R$ is the Cayley transform with respect to the root $\alpha$ as defined in Equation (41) of \cite{CEexp} and $\dim Q=\dim (M\cdot \Ad(vs_{\alpha})\fb_{\fr})$.   

\end{enumerate}

\end{prop}
\noindent Note that in all cases, $\pi_{\alpha}^{-1}(\pi_{\alpha}(Q))=MvP_{\alpha}\cdot \fb_{\fr}.$

\subsection{Extended Monoid action on $S_{i}$-orbits}
The general theory of Section \ref{ss:monoidbackground} can be applied to the setting of $G_{\Delta}$-orbits on 
$\B_{n}\times\B_{n}\times\mathbb{P}^{n-1}$ and $S_{i}$-orbits on $\B_{n}$.  The $G_{\Delta}$-orbits on $\B_{n}\times\B_{n}\times\mathbb{P}^{n-1}$ come equipped with a monoid action by $\Pi_{\fg}\times\Pi_{\fg}$, one copy of $\Pi_{\fg}$ for each factor of $\B_{n}$.  
We first consider the monoid action with respect to the second 
factor of $\B_{n}$ in the product $\B_{n}\times\B_{n}\times\mathbb{P}^{n-1}$.  
Any $G_{\Delta}$-orbit on $\B_{n}\times\B_{n}\times\mathbb{P}^{n-1}$ can be written in the form:
\begin{equation}\label{eq:Gdeltaorbit}
\mathcal{O}_{\Delta,\fb^{\prime}}:=G_{\Delta}\cdot (\fb, \fb^{\prime}, [e_{i}]), \mbox{ for some } i\in\{1,\dots, n\}.
\end{equation}
 Recall the orbit correspondences from Equations (\ref{eq:Gdeltacorres})-(\ref{eq:Sicorres}),
\begin{equation}\label{eq:three}
\mathcal{O}_{\Delta,\fb^{\prime}}=G\cdot (\fb, \fb^{\prime}, [e_{i}])\longleftrightarrow \mathcal{O}_{B,\fb^{\prime}}:=B\cdot (\fb^{\prime}, [e_{i}])\longleftrightarrow Q_{\fb^{\prime}}:=S_{i}\cdot \fb^{\prime}.
\end{equation}
Let $\alpha\in\Pi_{\fg}$. It follows from definitions that 
\begin{equation*}
\begin{split}
&\ms*_{2} \mathcal{O}_{\Delta,\fb^{\prime}} \mbox{ is open } G_{\Delta}-\mbox{orbit in } 
G_{\Delta}\cdot (\fb, P_{\alpha}^{\prime}\cdot \fb^{\prime}, [e_{i}])\longleftrightarrow \\
&\mbox{ open } B-\mbox{orbit in } B\cdot (P_{\alpha}^{\prime}\cdot\fb^{\prime}, [e_{i}])\longleftrightarrow \mbox{ open } S_{i}-\mbox{orbit in } S_{i}P_{\alpha}^{\prime}\cdot\fb^{\prime}. 
\end{split}
\end{equation*}
But the last orbit is exactly $\ms*Q_{\fb^{\prime}}$ from our discussion in Section \ref{ss:monoidbackground}.  
Summarizing, 
\begin{equation}\label{eq:secondfactormonoid}
\ms*_{2} \mathcal{O}_{\Delta,\fb^{\prime}}\longleftrightarrow \ms* Q_{\fb^{\prime}}=  \mbox{ open } S_{i}-\mbox{orbit in } S_{i}P_{\alpha}^{\prime}\cdot\fb^{\prime}.
\end{equation}

We now consider the monoid action on $G_{\Delta}\backslash(\B_{n}\times\B_{n}\times\mathbb{P}^{n-1})$ coming from the first factor of $\B_{n}$ in the triple product 
$\B_{n}\times\B_{n}\times\mathbb{P}^{n-1}$.  By definition,
\begin{equation}\label{eq:firstfactorfirst}
\ms*_{1} \mathcal{O}_{\Delta,\fb^{\prime}}=\mbox{ open } G_{\Delta}-\mbox{orbit in } G_{\Delta}\cdot (P_{\alpha}\cdot\fb, \fb^{\prime}, [e_{i}])\longleftrightarrow \mbox{ open } B-\mbox{orbit in } P_{\alpha}\cdot (\fb^{\prime}, [e_{i}]).
\end{equation}
For arbitrary $\alpha\in \Pi_{\fg}$, the parabolic subgroup $P_{\alpha}$ need not preserve the orbit $\mathcal{O}_{i}=B\cdot[e_{i}]$ and therefore $\ms*_{1}\mathcal{O}_{\Delta, \fb^{\prime}}$ need not correspond to an $S_{i}$-orbit on $\B_{n}$.  However, if $\alpha\in \{\alpha_{1},\dots, \alpha_{i-2},\alpha_{i+1},\dots, \alpha_{n-1}\}$, then $P_{\alpha}\cdot[e_{i}]=B\cdot[e_{i}]$ and $\ms*_{1} \calO_{\Delta,\fb^{\prime}}$ corresponds to an $S_{i}$-orbit on $\B_{n}$ which we denote by $\ms*_{L} Q_{\fb^{\prime}}$.  It will help to have a more detailed description of $\ms*_{L} Q_{\fb^{\prime}}$.
We define
\begin{equation}\label{eq:extraroots}
\mathfrak{S}_i:=\{\alpha\in\Pi_{\fg}|\; P_{\alpha}\cdot[e_{i}]=B\cdot[e_{i}]\}=\{\alpha_{1},\dots, \alpha_{i-2},\alpha_{i+1},\dots, \alpha_{n-1}\}.
\end{equation}
For any $p\in P_{\alpha}$, $p\cdot[e_{i}]=\tilde{b}\cdot[e_{i}]$ for some $\tilde{b}\in B$.  This implies that $\tilde{b}^{-1} p\in\mbox{Stab}_{P_{\alpha}}[e_{i}]$.  Let $s:=\tilde{b}^{-1} p\in\mbox{Stab}_{P_{\alpha}}[e_{i}]$, so that 
$p=\tilde{b}s$.  Therefore,
\begin{equation}\label{eq:secondhelper}
P_{\alpha}\cdot (\fb^{\prime}, [e_{i}])=B\cdot ( \mbox{Stab}_{P_{\alpha}}[e_{i}]\cdot\fb^{\prime}, [e_{i}]). 
\end{equation}
But the $B$-orbits in the variety in (\ref{eq:secondhelper}) correspond to $S_{i}$-orbits in the variety $(\mbox{Stab}_{P_{\alpha}}[e_{i}])\cdot\fb^{\prime}$.  Thus,  
\begin{equation}\label{eq:leftmonoidSi}
\mbox{ for } \alpha\in\mathfrak{S}_i:\; \;\ms*_{1} \mathcal{O}_{\Delta,\fb^{\prime}}\longleftrightarrow \ms*_{L}Q_{\fb^{\prime}}:=\mbox{ open } S_{i}-\mbox{orbit in } \mbox{Stab}_{P_{\alpha}}[e_{i}] \cdot\fb^{\prime}.
\end{equation}

\begin{dfn}\label{d:leftandright}
We refer to the monoid action by simple roots $\Pi_{\fg}$ on $S_{i}\backslash\B_{n}$ given in (\ref{eq:secondfactormonoid}) as the \emph{right monoid} action on $S_{i}$-orbits.  We refer to the monoid action given in Equation (\ref{eq:leftmonoidSi}) as the \emph{left monoid} action on $S_{i}$-orbits. Taken together, we refer to the monoid action by $\mathfrak{S}_i\coprod \Pi_{\fg}$ as the \emph{extended monoid action}.  
\end{dfn}

\begin{rem}\label{r:extended}
Recall from Remark \ref{r:specialcase} that when $i=n$, the group $S_{n}$ coincides up to centre with the standard upper triangular Borel subgroup $B_{n-1}$ of $GL(n-1)$.  In this case,
the roots $\mathfrak{S}_i=\{\alpha_{1},\dots,\alpha_{n-2}\}$ are the standard simple roots $\Pi_{\fgl(n-1)}$ of the subalgebra $\fgl(n-1)\subset\fg$.  The group $\mbox{Stab}_{P_{\alpha}}[e_{n}]$ coincides up to centre with the standard parabolic subgroup $P^{n-1}_{\alpha}$ of $GL(n-1)$ corresponding to the root $\alpha$.  Thus, in this case the left monoid action of (\ref{eq:leftmonoidSi}) coincides with the left monoid action via roots of $\fgl(n-1)$ on $\Borbitspace$ defined in Section 4 of \cite{CE21I}.  Taken together, Equations (\ref{eq:secondfactormonoid}) and (\ref{eq:leftmonoidSi}) coincide with the extended monoid action by simple roots $\Pi_{\fgl(n-1)\oplus\fg}$ on $\Borbitspace$ constructed in \emph{loc. cit.}.
\end{rem}

\subsection{The extended monoid action and $i$-Shareshian pairs}
In this section, we generalize Section 4.3 of \cite{Shpairs} to $S_{i}$-orbits on $\B_{n}$ and $i$-Shareshian pairs for any $i\in\{1,\dots, n\}$.  In \emph{loc. cit.}, we discuss how to obtain both left and right monoid actions on $B\backslash\B_{n}$ via simple roots of $\Pi_{\fg}$ and how these monoid actions induce left and right monoid actions of $W$ on itself.  These monoid actions are described in Equation (4.4) of \cite{Shpairs}, which we restate here for the convenience of the reader.
\begin{equation}\label{eq:Wmonoidact}
\begin{split}
&\mbox{ Let } \alpha\in\Pi_{\fg},\, w\in W;\\
\mbox{(Left Action) } &m(s_{\alpha})*_{L} w=w \mbox{ if } \ell(s_{\alpha}w)<\ell(w) \mbox{ and } m(s_{\alpha})*_{L}w=s_{\alpha}w\mbox{ if } \ell(s_{\alpha}w)>\ell(w).\\
\mbox{(Right Action) } &m(s_{\alpha})*_{R}w=w \mbox{ if } \ell(ws_{\alpha})<\ell(w) \mbox{ and } m(s_{\alpha})*_{R}w=ws_{\alpha}\mbox{ if } \ell(ws_{\alpha})>\ell(w).\\
\end{split}
\end{equation}
Recall that the Borel subalgebra $\fb^{i}$ stabilizes the flag $\mathcal{E}^{i}$ of Equation (\ref{eq:Fiflag}) and $\fb^{i}=\sigma_{i}(\fb),$ where $\sigma_{i}=(i,i-1,\dots, 2, 1)$ is the $i$-cycle from Proposition \ref{p:Shpairs}.   Let $\Pi^{i}_{\fg}=\{\alpha_{1}^{i},\dots, \alpha_{n-1}^{i}\}$ be the simple roots defined by the Borel subalgebra $\fb^{i}$ with corresponding set of simple reflections $T^{i}:=\{s^{i}_{1},\dots, s^{i}_{n-1}\}$.  It follows that $\alpha^{i}_{k}=\sigma_{i}(\alpha_{k})$ for $k\in \{1,\dots, n-1\}$.
\begin{nota}\label{n:Worders}
By considering the simple roots $\Pi_{\fg}^{i}$ and corresponding simple reflections $T^{i}$, we obtain a new order relation on $W$ by considering the length of an element $w^{i}\in W$ with respect to the set of simple reflections $T^{i}$.  We denote $W$ with this non-standard order relation by $(W, T^{i})$, and denote $W$ with the standard order relation by $(W,T)$.  We denote elements in the poset $(W,T)$ by $w$ and denote elements in the poset $(W, T^{i})$ by $w^{i}$.
\end{nota}
We also obtain left and right monoid actions of $(W,T^{i})$ on itself analogous to the ones in Equation (\ref{eq:Wmonoidact}) with $\alpha\in\Pi_{\fg}$ replaced by $\alpha^{i}\in \Pi_{\fg}^{i}$.  We now use the monoid actions of $(W,T)$ and $(W,T^{i})$ on themselves to define a monoid action of $\mathfrak{S}_i\coprod\Pi_{\fg}$ on $(W,T)\times (W,T^{i})$.  First, observe that $\mathfrak{S}_i\subset\Pi_{\fg}^{i}$.  Indeed, $\alpha_{k}=\alpha_{k+1}^{i}$ for $k=1,\dots, i-2$ and $\alpha_{k}=\alpha_{k}^{i}$ for $k=i+1,\dots, n-1$.   The following definition generalizes Definition 4.5 in \cite{Shpairs} to the case of arbitrary $i\in\{1,\dots, n\}$.

\begin{dfn}\label{d:diagonal}
Define the \emph{restricted diagonal monoid action} on $(W,T)\times (W,T^{i})$ via simple roots $\mathfrak{S}_i\coprod \Pi_{\fg}^{i}$ as follows. 

  
For $(w,u^{i})\in (W,T)\times (W,T^{i})$, define
\begin{equation}\label{eq:diagonalmonoid}
\begin{split}
\mbox{ (Left action)} &\mbox{ For } \alpha\in\mathfrak{S}_i,\, m(s_{\alpha})*_{L}(w,u^{i})=(m(s_{\alpha})*_{L}w, m(s_{\alpha})*_{L}u^{i}).\\
\mbox{ (Right action)} &\mbox{ For } \alpha\in\Pi_{\fg}, m(s_{\alpha})*_{R}(w,u^{i})=(m(s_{\alpha})*_{R}w, m(s_{\alpha^{i}})*_{R}u^{i}),
\end{split}
\end{equation}
where the monoid actions $*_{L}$ and $*_{R}$ are given in Equation (\ref{eq:Wmonoidact}).
\end{dfn}
Henceforth, we will drop the subscripts $L$ and $R$ on the monoid actions defined in Equation (\ref{eq:diagonalmonoid}) and use the convention that a simple root $\alpha\in\mathfrak{S}_i\coprod \Pi_{\fg}$ acts on the left whenever $\alpha\in \mathfrak{S}_i$ and on the right if $\alpha\in\Pi_{\fg}$.  We now arrive at the main result of this section, which generalizes Theorem 4.6 of \cite{Shpairs} to $S_{i}\backslash\B_{n}$ for arbitrary $i$.

\begin{thm}\label{thm:intertwine}(cf. Theorem 4.6 of \cite{Shpairs})
The $i$-Shareshian map $\Sh_{i}: S_{i}\backslash\B_{n} \to \Sp_{i}\subset W\times W$ given in Equation (\ref{eq:Shmap})
intertwines the monoid action by $\mathfrak{S}_i\coprod \Pi_{\fg}$ on $S_{i}\backslash\B_{n}$ given in Equations (\ref{eq:leftmonoidSi}) and (\ref{eq:secondfactormonoid}) with the restricted diagonal monoid action on $(W,T)\times (W,T^{i})$ given in (\ref{eq:diagonalmonoid}), i.e., for $Q\in S_{i}\backslash\B_{n}$ and $\alpha\in\mathfrak{S}_i\coprod \Pi_{\fg}$,
 \begin{equation}\label{eq:Shinter}
 \Sh_{i}(m(s_{\alpha})*Q)=m(s_{\alpha})*\Sh_{i}(Q).  
 \end{equation}
In particular, the restricted diagonal monoid action on $(W,T)\times (W, T^{i})$ preserves the set $\Sp_{i}$ of all $i$-Shareshian pairs.  
Moreover, if $m(s_{\alpha})*Q\neq Q$, then the type of the root $\alpha$ is determined by the type of $\alpha$ for the corresponding $i$-Shareshian pair $\Sh_{i}(Q)=(w,u^{i})$.  More precisely, for a root $\alpha\in\mathfrak{S}_i\coprod \Pi_{\fg}$, 
\begin{enumerate}
\item The root $\alpha$ is complex stable for $Q$ if and only if it is complex stable for both $w$ and $u^{i}$.
\item The root $\alpha$ is non-compact for $Q$ if and only if $\alpha$ is complex stable for exactly one of $w$ and $u^{i}$ and unstable for the other. 
\item The root $\alpha$ is real or complex unstable for $Q$ if and only if $\alpha$ is complex unstable for both $w$ and $u^{i}$.
\end{enumerate}
\end{thm}
\begin{proof}
The proof of the Theorem for the right monoid action via a root $\alpha\in\Pi_{\fg}$ is completely analogous to the corresponding proof in \cite{Shpairs}, so we omit the details here.  For the left action of a root $\alpha\in \mathfrak{S}_i$, the proof is largely analogous to the one in \cite{Shpairs} except that the proof of the analogue of Equation (4.7) of \emph{loc.cit.} is slightly different in this case.  Let $\Sh(Q)=(Q_B, Q_{B^{i}})\in B\backslash\B_{n}\times B^{i}\backslash\B_{n}$ so that $Q=Q_B \cap Q_{B^{i}}$ by Remark \ref{r:geoshar}.
We claim that for $\alpha\in \mathfrak{S}_i$,  
\begin{equation}\label{eq:leftmsainter}
\ms*Q=\ms*Q_{B}\cap\ms*Q_{B^{i}},
\end{equation}
Let $Q=S_{i}xB/B$ for some $x\in G$.  Then by Equation (\ref{eq:leftmonoidSi}), we know that  $\ms*Q$ is $\mbox{the open } S_{i}-\mbox{orbit in } Stab_{P_{\alpha}}[e_{i}] x B/B$.  We claim that 
\begin{equation}\label{eq:firstinc}
\ms* Q\subset \ms* Q_{B}.
\end{equation}
Note that $\ms*Q_{B}$ is $\mbox{the open } B-\mbox{orbit in } P_{\alpha}xB/B$ and that 
$\mbox{Stab}_{P_{\alpha}}[e_{i}] xB/B\subset P_{\alpha}xB/B$.  We first show that 
\begin{equation}\label{eq:theyintersect}
\ms*Q_{B}\cap \mbox{Stab}_{P_{\alpha}}[e_{i}] xB/B\neq\emptyset. 
\end{equation}
If $\alpha$ is complex unstable for $Q_{B}$, the assertion is clear.  If on the other hand, 
$\alpha$ is complex stable for $Q_{B}$, then by Proposition \ref{p:stableandnc}, we have 
$\ms*Q_{B}=Bs_{\alpha}xB/B$.  Since $\alpha\in\mathfrak{S}_i$, we can choose a representative ${\dot{s}}_{\alpha}$ of $s_{\alpha}$ contained in
$\mbox{Stab}_{P_{\alpha}}[e_{i}]$.  It then follows that 
${\dot{s}}_{\alpha} xB/B\in \ms*Q_{B}\cap \mbox{Stab}_{P_{\alpha}}[e_{i}] xB/B$.  Since $\alpha$ is complex for $Q_{B}$ these two cases show that the assertion in Equation 
(\ref{eq:theyintersect}) holds.  By Equation (\ref{eq:theyintersect}), $\ms*Q_{B} \cap \mbox{Stab}_{P_{\alpha}}[e_{i}] xB/B$ is open and dense in the irreducible variety $\mbox{Stab}_{P_{\alpha}}[e_{i}] xB/B,$ so $\ms*Q$ intersects $\ms*Q_B$, which implies Equation (\ref{eq:firstinc}).
One can repeat the above argument with $Q_{B}$ replaced by $Q_{B^{i}}$ by observing that 
$\mbox{Stab}_{P_{\alpha}}[e_{i}]\subset P_{\alpha}^{i}$, where $P_{\alpha}^{i}\supset B^{i}$ is the parabolic subgroup of $G$ corresponding to $\alpha\in\mathfrak{S}_i\subset \Pi_{\fg}^{i}$.  We leave the details to the reader.  Equation (\ref{eq:leftmsainter}) then follows from the inclusion $\ms*Q\subset \ms*Q_{B}\cap \ms*Q_{B^{i}}$ and Remark \ref{r:geoshar}.  

The proof of the statements (1)-(3) of the Theorem for the left monoid action via a root $\alpha\in\mathfrak{S}_i$ is analogous to the proof of the corresponding statements for the case of the left monoid action via a root $\alpha\in\Pi_{\fk}$ in Theorem 4.6 of \cite{Shpairs} and is omitted.

\end{proof}


\begin{rem}\label{r:exoticmonoid}
There are two other monoid actions on $S_{i}\backslash\B_{n}$ that were not discussed in the above discussion, since the analysis of these actions requires different techniques.  First let $\alpha=\alpha_{i-1}.$  Then we can define an additional left monoid action $=m(s_{i-1})*_{L}$ on the $S_{i}$-orbits on $\B_{n}$ using the fact that the minimal parabolic subgroup $P_{\alpha}$ corresponding to $\alpha$ has the property that $B\cdot [e_i]$ is open in $P_{\alpha}\cdot [e_i].$   We can prove that for a $S_i$-orbit $Q_{\fb^{\prime}} = S_i \cdot \fb^{\prime}$ in $\B_n$,
\begin{enumerate}
\item $Q_{\fb^{\prime}}$ is not complex stable for $\alpha$;
\item If $\Sh_{i}(Q_{\fb^{\prime}})=(w, u^{i})$, then 
 \begin{equation}\label{eq:firstSh}
 \Sh_{i}(m(s_{i-1})*_{L} Q_{\fb^{\prime}})=(m(s_{i-1})*_{L} w, u^{i}),
 \end{equation}
 where $m(s_{i-1})*_{L}w$ denotes the left action of $s_{i-1}$ on $w$ as in
 (\ref{eq:Wmonoidact}). 
\end{enumerate} It follows that $\alpha$ is non-compact for $Q_{\fb^{\prime}}$ if and only if $\alpha$ is complex stable for $w$.
There is also a left monoid action on $S_{i}\backslash\B_{n}$ by the simple root 
 $\alpha_{1}^{i}=\eps_{i}-\eps_{1}\in\Pi_{\fg}^{i}$.  The construction of this monoid action will be explained in the sequel.  Results analogous to those for the action by $m(s_{i-1})$ hold for this action by $m(s_{i}^{1})$, except that in Equation (\ref{eq:firstSh}), the left action by $m(s_1^{i})$ is on the second factor instead of on the first factor.  We will prove these assertions in the sequel.\end{rem}

 \begin{exam}\label{ex:S2orbits}
 We conclude this paper by explicitly describing the Bruhat graph for the $S_{2}$-orbits on $\B_{3}$.  We use the language of standardized $2$-Shareshian pairs to describe $S_{2}$-orbits (see Definition \ref{d:standardizedpair}). By Example \ref{ex:formulatable}, $|S_{2}\backslash\B_{3}|=9$.  In the graph below, the zero dimensional orbits appear on the first row, the one dimensional orbits on the second, etc.     We also indicate the monoid actions.  A red line denotes a non-compact root and a blue line denotes a complex stable root.  Right actions are denoted by a solid line and left action by a dashed line.  A solid green line indicates a closure relation that is not obtained from a monoid action.  Note that in this case the set $\mathfrak{S}_i$ given in (\ref{eq:extraroots}) is empty, so the only left actions are the actions discussed in Remark \ref{r:exoticmonoid}.   For simplicity of notation, let $\alpha=\alpha_1$, let $\beta=\alpha_2$, let $s=s_{\alpha}$, and let $t=s_{\beta}.$   Then $\Pi_{\fgl(3)}^{2}=\{\alpha^{2},\beta^{2}\}$.   Note that the monoid action by $\alpha_1^{i}$ in Remark \ref{r:exoticmonoid} on a standardized $i$-Shareshian pair $(w, y)$
 is given by $m(s_{1}^{i})*(w,y)=(w, m(s_{\sigma_{i}^{-1}(\alpha_{1}^{i})})*_{L} y)=(w, m(s_{1})*_{L} y),$ where $\sigma_{i}$ is the $i$-cycle of Proposition \ref{p:Shpairs}.  Note also that for $i$-standardized Shareshian pairs, the $i$-Shareshian order is simply the restriction of the two-fold product of the standard Bruhat order to the set of standardized pairs (see discussion after Equation (\ref{eq:iShorder})). This Example illustrates the assertion of Equation (\ref{eq:iShorder}) (to be proved in the sequel) for $2$-standardized Shareshian pairs in $\Sigma_{3}\times\Sigma_{3}$.
 

\begin{center}
\begin{tikzpicture}  
 [scale=2.0,auto=center,every node/.style={rectangle,fill=white!20}] 
\node (a1) at  (-1,5) {$(e, s)$};
\node (a2) at (3,5) {$(s,e) $};
\node (a3) at (-2,3) {$(t,st)$};
\node (a4) at (1,3) {$(s,s)$};
\node (a5) at (4,3) {$(st,t)$};
\node (a6) at (-2,1) {$(ts, sts)$};
\node (a7) at (1,1) {$(st,st)$};
\node (a8) at (4,1) {$(sts,ts)$} ;
\node (a9) at (1,-1) {$(sts,sts)$};
\draw [blue] (a1) -- (a3) node[midway, above] {$\beta$}; 
  \draw [red] (a1) -- (a4) node[midway, above] {$\alpha$};  
 \draw [red] (a2) -- (a4) node[midway, above] {$\alpha$};  
 \draw [blue] (a2)--(a5) node[midway, above] {$\beta$};
 \draw [blue] (a3)--(a6) node[midway, above] {$\alpha$};
  \draw [red][dashed] (a3)--(a7) node[near start, above] {$\alpha$} ;
  \draw[blue] (a4)--(a7) node[midway, above] {$\beta$};
  \draw[green] (a4)--(a8);
  \draw[green] (a4)--(a6);
  \draw[red][dashed] (a5)--(a7) node[near start, above] {$\alpha^{2}$} ;
  \draw[blue] (a5)--(a8) node[midway, above] {$\alpha$};
  \draw[red]  (a6)--(a9) node[midway, above] {$\beta$};
    \draw[blue]  (a7)--(a9) node[midway, above] {$\alpha$};
       \draw[red]  (a8)--(a9) node[midway, above] {$\beta$};
\end{tikzpicture}
\end{center}

 \end{exam}

\bibliographystyle{amsalpha.bst}

\bibliography{bibliography-1}

\end{document}